\documentclass[12pt]{amsart}

\DeclareMathOperator{\Lie}{Lie}

\DeclareMathOperator{\Hom}{Hom}

\DeclareMathOperator{\id}{id}

\DeclareMathOperator{\ad}{ad}

\DeclareMathOperator{\ssc}{sc}

\DeclareMathOperator{\Int}{Int}
\DeclareMathOperator{\Aut}{Aut}

\DeclareMathOperator{\Gal}{Gal}

\DeclareMathOperator{\Pin}{Pin} 
\DeclareMathOperator{\Spin}{Spin} 
\DeclareMathOperator{\End}{End} 
 
\DeclareMathOperator{\Br}{Br}

\DeclareMathOperator{\inv}{inv}

\DeclareMathOperator{\tr}{tr}
\DeclareMathOperator{\Wall}{Wall}

\DeclareMathOperator{\fppf}{fppf}
\DeclareMathOperator{\disc}{disc}

\usepackage{amscd}
\usepackage{amssymb}
\usepackage{hhline}

\numberwithin{equation}{section}

\newtheorem{theorem}{Theorem}[section]
\newtheorem{corollary}[theorem]{Corollary}
\newtheorem{lemma}[theorem]{Lemma}
\newtheorem{proposition}[theorem]{Proposition}

\theoremstyle{definition}
\newtheorem{definition}[theorem]{Definition} 
\newtheorem{remark}[theorem]{Remark}

\usepackage{hyperref}

\begin{document}

\title[Reductive groups,  epsilon factors and Weil indices]
{Reductive groups, epsilon factors and Weil indices}

\author{Robert\ E.\ Kottwitz}
\address{Robert E. Kottwitz\\Department of Mathematics\\ University of Chicago\\
5734 University
Avenue\\ Chicago, Illinois 60637}

\email{kottwitz@math.uchicago.edu}



\maketitle 

\section{Introduction }

\subsection{Notation pertaining to the ground field $F$}

We work over a locally compact field $F$, assumed to be nondiscrete, 
and we fix a nontrivial character $\psi: F \to \mathbb C^\times$ on the additive
group $F$.
We make no restrictions on the characteristic of $F$. 
Finally we choose a separable algebraic closure $\overline{F}$ of $F$ and put
$\Gamma := \Gal(\overline{F}/F)$.

\subsection{Notation pertaining to the group $G$} 

We consider a connected reductive group $G$ over $F$. We write $G_{\ad}$ for its
adjoint group and
$G_{\ssc}$ for the simply connected cover of the derived group of $G$. There are
then
canonical homomorphisms 
\[
G_{\ssc} \to G \to G_{\ad}. 
\]

We choose a quasi-split inner form $G_0$ of $G$. We also choose a maximal $F$-torus
$T_0$ in $G_0$ having
the property that there exists a Borel $F$-subgroup $B_0$ containing it. It is
well-known that $T_0$ is unique up
to conjugacy under $G_0(F)$. 

\subsection{Three ingredients in the main theorem}

\subsubsection{First ingredient} 

The first ingredient  is the sign $e(G) \in \{\pm 1 \}$ attached to $G$ in 
\cite{SignChanges}. 

Now let $T$ be a maximal $F$-torus in $G$. 
The remaining two ingredients are complex fourth roots of unity that depend on
$(G,T)$.

\subsubsection{Second ingredient}

The Galois group $\Gamma$ acts on the cocharacter group $X_*(T)$. Of course the
action is trivial on some open
subgroup of $\Gamma$. In any event $X_*(T)_\mathbb R$ is a real representation 
of $\Gamma$, and $X_*(T)_\mathbb C$ is a complex orthogonal representation of
$\Gamma$.

We are interested in the complex orthogonal  representations 
$X_*(T)_\mathbb C$ and 
$X_*(T_0)_\mathbb C$. Their difference 
$X_*(T)_\mathbb C - X_*(T_0)_\mathbb C$ is a
virtual orthogonal representation  whose dimension is $0$. 
The second ingredient  is the local epsilon factor 
\[
\epsilon_L \bigl(X_*(T)_\mathbb C - X_*(T_0)_\mathbb C, \psi \bigr).
\]
Here $\epsilon_L$ is as in Tate's article \cite{Tate}: 
the subscript $L$ indicates that epsilon factors are
formed using the conventions of Langlands rather than those 
of Deligne. It would make no difference if
we used $X^*(T)$ rather than $X_*(T)$, since we are dealing with
self-contragredient representations.

Deligne's results \cite{Del} on  local epsilon factors of virtual  orthogonal 
complex representations of $\Gamma$ are essential to our proof  
of the main theorem, as is  a result of Jacquet-Langlands  \cite{JL}  
giving an equality between certain local epsilon factors  
and certain Weil indices. 

\subsubsection{Third ingredient} \label{subsub.ThirdIngredient}

The third ingredient is the Weil index $\gamma(Q_V,\psi)$ (see \cite{Weil}) of a certain
even dimensional, nondegenerate quadratic space $(V,Q_V)$ 
over $F$. Because the dimension is even, the Weil index 
is a complex fourth root of unity.

The quadratic form $Q_V$ depends on $(G,T)$, and it lives on the $F$-vector space
$V$
defined as the unique $T$-invariant complement to $\mathfrak t: =\Lie(T)$ in
$\mathfrak g := \Lie(G)$.
Equivalently, $V$ is the unique $T_{\ssc}$-invariant complement to $\mathfrak
t_{\ssc}$
in $\mathfrak g_{\ssc}$. (As is customary, $T_{\ssc}$ denotes the preimage of $T$
in $G_{\ssc}$.)

How is the quadratic form $Q_V$ defined? Giving an explicit formula for it is
possible only over the
separable closure $\overline{F}$. So, just for a moment, we extend scalars to
$\overline{F}$.
We then have the root space decomposition 
\[
V = \bigoplus_\alpha \mathfrak g_\alpha.
\]
Even more useful to us  is the coarser decomposition 
\[
V = \bigoplus_{\pm\alpha} \mathfrak g_{\pm\alpha},
\]
where $\mathfrak g_{\pm\alpha}:=\mathfrak g_\alpha \oplus \mathfrak g_{-\alpha}$. 
Our quadratic form $Q_V$ will be defined as the direct sum of quadratic forms
$Q_{\pm\alpha}$ on
the planes $\mathfrak g_{\pm\alpha}$ over $\overline{F}$. 

The value of $Q_{\pm\alpha}$ on $v \in \mathfrak g_{\pm\alpha}$ 
is defined as follows. We decompose $v$ as the sum of $v_+ \in \mathfrak g_\alpha$
and
$v_- \in \mathfrak g_{-\alpha}$. We regard $v_+$ and $v_-$ as elements in
$\mathfrak g_{\ssc}$,
and then form their Lie bracket $[v_+,v_-] \in \mathfrak t_{\ssc}$. It is a
standard
fact that $[v_+,v_-]$ is a scalar multiple of the coroot for $\alpha$, and we take
this
scalar (in $\overline{F}$) as the definition of $Q_{\pm\alpha}(v)$. (Even when the
characteristic is $2$, the
coroot $H_\alpha$ is nonzero in $\mathfrak t_{\ssc}$, which is why we moved from 
$\mathfrak g$ to $\mathfrak g_{\ssc}$ before taking the Lie bracket.)

 In summary, we 
define $Q_{\pm\alpha}(v)$ to be the unique scalar $c \in \overline{F}$ such that 
\[
[v_+,v_-]_{\ssc}= c H_\alpha, 
\]
where the subscript $\ssc$ indicates that the Lie bracket is taken in $\mathfrak
g_{\ssc}=\mathfrak t_{\ssc} \oplus V$,
and $H_\alpha \in \mathfrak t_{\ssc}$ is the coroot for $\alpha$. 
If we replace $\alpha$ by $-\alpha$, both sides 
of this equality are replaced by their negatives, but $c$ does not change. So
$Q_{\pm\alpha}$ depends only on the
unordered pair $\{\alpha,-\alpha\}$. It is a canonically defined nondegenerate
quadratic form on the
plane $\mathfrak g_{\pm\alpha}$.  

Now $Q_{\pm\alpha}$ is defined  only over $\overline{F}$, but 
the direct sum 
\[
Q_V := \bigoplus_{\pm\alpha} Q_{\pm\alpha}
\] 
is defined over $F$, and our discussion 
of the third ingredient is complete.

\subsection{Statement of the main theorem}

For any $(G,T)$ over $F$, we want to prove the following statement, an application
of which
can be found in Kaletha's paper \cite{Kal}. 
\begin{theorem} There is an equality 
\[
e(G) \gamma(Q_V,\psi) = \epsilon_L \bigl(X_*(T)_\mathbb C - X_*(T_0)_\mathbb C,\psi
\bigr).
\]
\end{theorem} 

 \subsection{Organization of the paper}

In what follows we first sketch a 
proof of the main theorem, and then fill in the details. The body 
of the paper is not entirely self-contained, and the reader will need 
to be familiar with the material in the appendices, where some 
of the key notation is introduced. 

The appendices are largely review, so many proofs are 
omitted. However Proposition \ref{thm.CompEpWeil}, which gives 
a convenient method to recognize when Weil 
indices and local epsilon factors of orthogonal representations 
are equal, may be new. The same is true of Proposition 
\ref{prop.FrohExt} for fields of characteristic $2$. For fields of 
characteristic different from $2$, Proposition \ref{prop.FrohExt} 
is a result of Fr\"ohlich \cite{Froh}.

\section{Sketch of the proof of the main theorem}

\subsection{Reduction steps} 

We refer to the statement in the theorem as $P(G,T)$ in the reduction steps that
follow.

\subsubsection{Reduction 1} 

The statement $P(G,T)$ is true if and only if the statement $P(G_{\ad},T_{\ad})$ is
true.
Indeed, each of three quantities appearing in the statement remains unchanged when
we pass from
$(G,T)$ to $(G_{\ad},T_{\ad})$. 
(As is customary $T_{\ad}$ is the maximal torus in $G_{\ad}$ 
whose preimage in $G$ is $T$.)

\subsubsection{Reduction 2}

If $P(G_i,T_i)$ is true for $i=1,2$, then $P(G_1 \times G_2,T_1 \times T_2)$ is
true.
Indeed, each of the three quantities appearing in the statement is multiplicative
in the
pair $(G,T)$. 

\subsubsection{Reduction 3}

Let $E/F$ be a finite separable extension, and use $\psi \circ \tr_{E/F}$ as
additive character on $E$.
Then, for any pair $(G,T)$ over $E$,  statement $P(G,T)$ is true if and only if  
statement $P(R_{E/F}G,R_{E/F}T)$ is true. Here $R_{E/F}$ denotes Weil restriction
of scalars.
Indeed, each of the three relevant quantities is preserved by Weil restriction of
scalars. (We are using
that local epsilon factors are inductive for virtual representations of degree
$0$.)

\subsubsection{Reduction 4}

 It is well-known that $T$ transfers to 
a maximal torus $T'$ in the quasi-split inner form $G_0$. Then, for any choice of
$T'$ (which is only well-defined up to
stable conjugacy), the statement $P(G,T)$ is true if and only if the statement
$P(G_0,T')$ is true.

This reduction is a bit more complicated than the others. 
The local epsilon factor is of course the same for 
$(G,T)$ and $(G_0,T')$, and moreover $e(G_0)=1$. So we need to show that 
\[
e(G)  =  \gamma(Q_{V'},\psi) \gamma(Q_V,\psi) ^{-1}.
\]
This statement is similar to a result of Gan \cite{Gan}
relating $e(G)$ to  the Killing form. 

Now the difference between $T$ and $T'$ is measured by an invariant $t \in
H^1(F,T)$.
By Corollary \ref{cor.EpTor} we just need to show that $e(G) \in \Br_2(F) = \{ \pm
1 \}$ is equal to
the image of $t$ under the map 
\[
\partial_U : H^1(F,T) \to \Br_2(F)
\]
arising from the fppf short exact sequence 
\[
1 \to \mu_2 \to U \to T \to 1
\]
corresponding to the element $\lambda_V \in X^*(T)/2X^*(T)$ defined in the
discussion leading
up to Corollary \ref{cor.EpTor}. It is evident that $\lambda_V$ works out to the
sum of the positive roots
(for any choice of positive system), and therefore gives rise to the same
connecting map as the one
used to define $e(G)$. 

\subsection{Using the reduction steps to get down to the essential case} 

How do we use these reduction steps? Using the first one, we may assume that $G$ is
adjoint.
Then $G$ is a product of $F$-simple groups. By the second reduction we may now
assume
that $G$ is $F$-simple. Then $(G,T)$ is of the form $(R_{E/F}H,R_{E/F}T_H)$ for
some $E/F$ and some
$(H,T_H)$ over $E$ with $H$ absolutely simple. By the third reduction we may now
assume that $G$
is an absolutely simple group. Finally, the fourth reduction allows us to assume
that $G$ is also quasi-split
over $F$.

Because $G$ is quasi-split, $e(G)=1$ and the main theorem 
reduces to the statement that 
\begin{equation}
\gamma(Q_V,\psi) = 
\epsilon_L \bigl(X_*(T)_\mathbb C - X_*(T_0)_\mathbb C,\psi \bigr).
\end{equation}
To prove this it suffices to check 
 that the number 
\begin{equation}\label{eq.MTReform}
\gamma(Q_V,\psi)\epsilon_L(X_*(T)_\mathbb C,\psi)^{-1}
\end{equation}
is independent of the maximal torus $T$ in $G$. Indeed, 
when $T=T_0$ the number \eqref{eq.MTReform} reduces to 
$\epsilon_L(X_*(T)_\mathbb C,\psi)^{-1}$, because the Weil 
index of $V_0$ is automatically $1$. Here $V_0$ is the analog 
of $V$ for $T_0$. 
 (We have 
$V_0 = V_+ \oplus V_-$, where $V_+$ comes from $B_0$-positive roots and $V_-$ comes
from
$B_0$-negative roots. Both $V_+$ and $V_-$ are defined over $F$, and the quadratic
form $Q_{V_0}$
vanishes identically on both of them. So $V_0$ is a direct sum 
of hyperbolic planes, and  its Weil index is  $1$.)

\subsection{Rough sketch} 

We finish this section by giving a rough sketch of the rest of the proof. 
In subsequent sections we will fill in  all the details. 
We are now working with a quasi-split absolutely simple  adjoint group $G$.  
For any two maximal $F$-tori $T$, $\tilde T$ in $G$ we must  show that 
\[
\gamma(Q_{\tilde V},\psi)\gamma(Q_V,\psi) ^{-1} =  
\epsilon_L(X_*(\tilde T)_\mathbb C - X_*(T)_\mathbb C,\psi).
\]
Here, of course, $ \mathfrak g = \tilde{\mathfrak t} \oplus \tilde V$ is the analog
for $\tilde T$ of the decomposition
$ \mathfrak g = \mathfrak t \oplus V$.

How will we do this? The details depend on the type of the root system $R$ of $G$.
Types $A,D,E$
(the simply laced ones) are the easiest to handle. Types $B,C,F_4$ will require
extra effort in characteristic $2$, and
$G_2$ will require extra effort in characteristic $3$, but in other characteristics
the non-simply laced cases
 can be treated by a slight elaboration of the 
method used to handle types $A,D,E$.   

Curiously enough the group $E_8$ is easiest of all: besides being simply laced, it
has the property that
the adjoint group is also simply connected. So we begin our sketch by considering
the case of $E_8$.
Then there is a canonical $G$-invariant nondegenerate quadratic 
form $Q_G$ on $\mathfrak g$ having the property that, 
for every maximal torus $T$ in $G$, the restriction of $Q_G$ to $V$ is equal to
$Q_V$. We then have
\[
\gamma(Q_G|_\mathfrak t,\psi)\gamma(Q_V,\psi) = \gamma(Q_G,\psi) = 
\gamma(Q_G|_{\tilde{\mathfrak t}},\psi)\gamma(Q_{\tilde V},\psi),
\]
from which we conclude that 
\[
\gamma(Q_{\tilde V},\psi)\gamma(Q_V,\psi) ^{-1} = \gamma(Q_G|_\mathfrak t,\psi) 
\gamma(Q_G|_{\tilde {\mathfrak t}},\psi)^{-1} .
\]
To complete the argument it remains only to show that  
\[ \gamma(Q_G|_\mathfrak t,\psi) 
\gamma(Q_G|_{\tilde{\mathfrak t}},\psi)^{-1} = 
\epsilon_L(X_*(\tilde T)_\mathbb C -
X_*(T)_\mathbb C,\psi).
\]
For this  (see Corollary \ref{cor.Wep}) we need to show that 
\[
HW(Q_G|_{\tilde{\mathfrak t}},Q_G|_\mathfrak t) = 
\overline{SW} \bigl(X_*(\tilde T)_\mathbb C - X_*(T)_\mathbb C).  
\]
A  general result of Fr\"ohlich (see subsection \ref{sub.FrohReview}) 
can be used to show that 
 the Hasse-Witt invariant on the left is equal to 
the Stiefel-Whitney class on the right, plus a correction 
term whose vanishing we must check. 
(In fact we will see that the correction term vanishes in all 
simply laced cases. In the non-simply laced cases it need not 
vanish, but it is canceled by another correction term
coming from the fact that the restriction of $Q_G$ to $V$ is 
no longer equal to $Q_V$.)

For other root systems there is no quadratic form $Q_G$ as well-behaved as the one
for $E_8$.
Even in type $A_n$, we don't really want to work with the adjoint group $PGL_n$,
and the simply connected
form $SL_n$ isn't much better. We could use $GL_n$, but we prefer a systematic 
construction that works perfectly for all simply laced groups and makes a
significant improvement even in
the non-simply laced case. 

As a result of our reduction process, we happened to end up with a group of adjoint
type. But, as we saw in the first reduction step, for general $(G,T)$ the truth of
$P(G,T)$ is the same as
that of $P(G_{\ad},T_{\ad})$. Since every maximal torus in $G_{\ad}$ is of the form
$T_{\ad}$ for
a unique maximal torus $T$ of $G$, we are free to replace the adjoint form by one
that is more convenient. Our choice
will be one that has been used in various settings, e.g.~ Vinberg's work \cite{V}
on reductive monoids.

So for the remainder of this sketch $G$ will no longer be adjoint. It will still be
quasi-split with
irreducible root system, and, because Vinberg's construction yields groups with
especially favorable
properties,  $\mathfrak g$ will admit a canonical $G$-invariant 
quadratic form $Q_G$ that is 
often nondegenerate. It will turn out that $Q_G$ is degenerate only 
for $B,C,F_4$ in characteristic $2$ and $G_2$ in 
characteristic $3$. 

In fact, now that we have moved from the adjoint form to 
a better behaved one, the approach sketched for $E_8$  
continues to work for all 
groups of type $A,D,E$. For the other types, in which there are long 
roots and short roots, it is no longer 
true that the restriction of $Q_G$ to $V$ coincides with $Q_V$. In 
fact, $V$ breaks up as a direct sum of 
two subspaces, one coming from the short roots, and one from the 
long roots. On the part coming from the long roots, 
$Q_G$ and $Q_V$  agree. On the part coming from the 
short roots, $Q_G$  agrees with $\ell Q_V$, 
where $\ell$ is $2$ in types $B,C,F_4$ and is $3$ in type $G_2$. 
As long as the characteristic of $F$ is 
different from $\ell$, a slight modification of the approach 
used for $A,D,E$ still works. 

When the characteristic of $F$ is equal to $\ell$, the 
quadratic form $Q_G$ is degenerate,
and the representation of $G$ on $\mathfrak g$ is reducible. 
In fact there is a canonical short exact sequence 
\[
0 \to \mathfrak g'' \to \mathfrak g \to \mathfrak g' \to 0 
\]
of $G$-modules, 
and we can achieve our goal by constructing nondegenerate quadratic forms
$Q_{\mathfrak g'}$,
$Q_{\mathfrak g''}$ on $\mathfrak g'$, $\mathfrak g''$ respectively, and then 
working with $(\mathfrak g',Q_{\mathfrak g'})$ and $(\mathfrak g'',Q_{\mathfrak
g''})$ rather than
$(\mathfrak g,Q_G)$. 

This completes our sketch. Now we resume giving a detailed proof. We start with
some preliminary
constructions over $\mathbb Z$. 

\section{Discussion of the pinned group $\mathbb G$ over $\mathbb Z$} 

As we have seen, it is enough to treat quasi-split adjoint groups with irreducible
root system. Such a group is
obtained by twisting a split adjoint group using a homomorphism  from $\Gamma$ to 
the group of automorphisms of the Dynkin diagram. But, as we mentioned above, we
are not obliged
to stick with adjoint groups, and in fact it is advantageous to use Vinberg's
construction,
which produces groups $G$ that are symmetrically located between $G_{\ssc}$ and
$G_{\ad}$.
 We also need to construct 
the relevant quadratic forms, and for this it is best to start with 
pinned groups over $\mathbb Z$. Only later we will extend scalars to $F$ and make
an outer twist.

\subsection{Definition of the group $\mathbb G$}

We begin with  a pinned  group 
\[
(\mathbb G_{\ssc},\mathbb B_{\ssc},\mathbb T_{\ssc},\{\eta_\alpha\})
\] 
over $\mathbb Z$ with $\mathbb G_{\ssc}$  semisimple 
and simply connected. We use the same system of notation as in Appendix
\ref{app.StQuadForm}, so the
reader should at least glance at that appendix before continuing to read this
section.

We 
assume that the root system $R$ of $\mathbb T_{\ssc}$ in $\mathbb G_{\ssc}$ is
irreducible. To $R$
is associated $\ell \in \{1,2,3\}$ (see Appendix \ref{app.StQuadForm}).

Consider the connected reductive group $\mathbb G$ over $\mathbb Z$ obtained by
putting
\[
\mathbb G := (\mathbb G_{\ssc} \times \mathbb T_{\ssc})/Z(\mathbb G_{\ssc}), 
\]
where $Z(\mathbb G_{\ssc})$ is the center of $\mathbb G_{\ssc}$, embedded in 
$(\mathbb G_{\ssc} \times \mathbb T_{\ssc})$ via $z \mapsto (z,z)$. 
We will also need the maximal torus $\mathbb T = (\mathbb T_{\ssc} \times \mathbb
T_{\ssc})/Z(\mathbb G_{\ssc})$.
The group $\mathbb G$  is useful for our purposes because, 
as we will see, there is an especially well-behaved quadratic form $Q_{\mathbb G}$
on the
free abelian group $\Lie(\mathbb G)$.

Observe that 
\begin{align} \label{X_*(T)} 
&X_*(\mathbb T)=\{(x_1,x_2) \in X_*(\mathbb T_{\ad}) \times
X_*(\mathbb T_{\ad})\,:\, x_1-x_2 \in X_*(\mathbb T_{\ssc})\},  \\ 
\label{X^*(T)} 
&X^*(\mathbb T)=\{(x_1,x_2) \in X^*(\mathbb T_{\ssc}) \times
X^*(\mathbb T_{\ssc})\,:\, x_1+x_2 \in X^*(\mathbb T_{\ad})\},
\end{align} 
where $\mathbb T_{\ad}$ is the maximal torus $\mathbb T_{\ssc}/Z(\mathbb G_{\ssc})$
in
the adjoint group $\mathbb G_{\ad}$ of $\mathbb G$. Here, as usual, we are using 
$\mathbb T_{\ssc} \to \mathbb T_{\ad}$ to identify the root lattice $X^*(\mathbb
T_{\ad})$ with a
subgroup of finite index in the weight lattice $X^*(\mathbb T_{\ssc})$, and
similarly for the coroot
lattice inside the coweight lattice. 

We also put
$\mathbb B:=(\mathbb B_{\ssc} \times 
\mathbb T_{\ssc})/Z(\mathbb G_{\ssc})$, a Borel
subgroup of $\mathbb G$.

\subsection{Definition of $Q_{\mathbb T}$} 

As in Appendix \ref{app.StQuadForm} we consider the $\mathbb Q$-vector space  
\[
\mathfrak a := X_*(\mathbb T_{\ssc})_{\mathbb Q} = X_*(\mathbb T_{\ad})_{\mathbb
Q},
\]
and the quadratic form $Q_1$ on it. 
We define a canonical quadratic form $Q_\mathbb T: \mathfrak a \times \mathfrak a
\to \mathbb Q$ by
putting
\[
Q_\mathbb T(x_1,x_2) := Q_1(x_1) - Q_1(x_2).
\] 
We write $B_\mathbb T$ for the symmetric bilinear form obtained by polarizing
$Q_\mathbb T$. So
$B_\mathbb T(x,x) = 2 Q_\mathbb T (x)$ and $Q_\mathbb T(x+y) = Q_\mathbb T(x) 
+ Q_\mathbb T(y) + B_\mathbb T(x,y)$.  We also view $B_\mathbb T$ 
 as an isomorphism from $\mathfrak a \times \mathfrak a$  
to its linear dual.  

The form $Q_\mathbb T$ is invariant 
under the natural action of $\Omega \times \Omega$ on $\mathfrak a \times \mathfrak
a$,
because $Q_1$ is invariant under $\Omega$. (See Appendix \ref{app.StQuadForm} for
the definition of
$\Omega = W \rtimes \Omega_0$.)

\subsection{Integrality properties of $Q_\mathbb T$}  

We see from equation
\eqref{X_*(T)} that $\Lambda:=X_*(\mathbb T)$ is a lattice in the  vector space
$\mathfrak a \times \mathfrak a$. We need to understand the integrality
properties of $Q_\mathbb T$ relative to  $\Lambda$. We denote by
$\Lambda^{\perp}$ the lattice perpendicular to $\Lambda$. By definition,
we have 
\[
\Lambda^\perp=\{ x \in \mathfrak a \times \mathfrak a : B_\mathbb T(x,x') \in
\mathbb Z \quad \forall \, x' \in \Lambda  \}.
\] Now $X^*(\mathbb T)$ is a lattice in
$\mathfrak a^* \times \mathfrak a^*$, and it follows immediately from the
definitions that
the isomorphism $B_\mathbb T$ maps $\Lambda^\perp$ onto $X^*(\mathbb T)$. 

\begin{lemma}\label{lem.QTint} The following statements hold. 
\begin{enumerate}
\item $Q_\mathbb T$ takes integral values on the lattice $\Lambda$. 
\item $\ell Q_\mathbb T$ takes integral values on the lattice $\Lambda^\perp$. 
\end{enumerate}
\end{lemma} 

\begin{proof} This follows easily from  Lemma \ref{lem.intQ}. 
\end{proof}

Now, as we already mentioned, the quadratic form $Q_\mathbb T$ is invariant under
$\Omega \times \Omega$.
The lattice $\Lambda= X_*(\mathbb T)$, however, 
is invariant only under $W\times W$ and the diagonal copy of 
$\Omega_0$. In
any event, it  makes sense to say that the quadratic form $Q_\mathbb T$ 
on $\Lambda$ is invariant under $(W \times W) \rtimes \Omega_0$. 

\subsection{Values of  $Q_\mathbb T$ on roots and coroots}
 The root lattice inside $X^*(\mathbb T)$ 
can of course be identified with  $X^*(\mathbb T_{\ad})$. The  
root $\alpha$ then becomes the pair $(\alpha,0)$ in the description \eqref{X^*(T)}
of $X^*(\mathbb T)$.
The situation for the coroot lattice is parallel, and the coroot $\alpha^\vee$  
becomes the pair $(\alpha^\vee,0)$ in the description \eqref{X_*(T)} of
$X_*(\mathbb T)$.

\begin{lemma}\label{lem.RootStuff}
There are equalities 
\begin{equation*}
Q_\mathbb T(\alpha^\vee)=\ell(\alpha^\vee), \quad 
 (\ell Q_\mathbb T)(\alpha) = \ell(\alpha). 
\end{equation*}
\end{lemma} 

\begin{proof}
We have $Q_\mathbb T(\alpha^\vee) = Q_1(\alpha^\vee) = \ell(\alpha^\vee)$.
Similarly,
 $(\ell Q_\mathbb T)(\alpha) = (\ell Q_1)(\alpha) = \ell(\alpha)$. 
In the first chain of equalities, we used the definition of $\ell(\alpha^\vee)$,
and in the second we used
 Lemma \ref{lem.ellalph}(2). 
\end{proof}

\begin{corollary}\label{cor.QTalpha}
 When $\ell = 1$, we have  $\alpha^\vee = \alpha$ and 
$Q_{\mathbb T}(\alpha^\vee) = Q_\mathbb T(\alpha) = 1$. 
When $\ell = 2,3$,  there are two root lengths, and we have 
\begin{itemize} 
\item 
$Q_{\mathbb T}(\alpha^\vee) = 1$ when $\alpha$ is long, and 
\item 
$(\ell Q_{\mathbb T})(\alpha) = 1$ when $\alpha$ is short. 
\end{itemize}
\end{corollary}
\begin{proof}
It follows from Lemma \ref{lem.ellalph}(1) that $\alpha^\vee=
\alpha$ when $\ell=1$. The rest follows from Lemma 
\ref{lem.RootStuff}. 
\end{proof}

\subsection{The quadratic form $Q_\mathbb G$ on $\Lie(\mathbb G)$}

In Appendix  \ref{app.StQuadForm}     
we studied a certain quadratic form 
$Q_2$ on $\Lie(\mathbb G_{\ssc}) = 
\Lie(\mathbb T_{\ssc}) \oplus \mathbb V$. 
In this subsection we 
extend $Q_2$ to a quadratic form $Q_\mathbb G$ on 
$\Lie(\mathbb G) = \Lie(\mathbb T) \oplus \mathbb V$.

Let $Q_\mathbb G$ be the unique quadratic form on $\Lie(\mathbb G)$ such that 
\begin{itemize}
\item
 $\Lie(\mathbb T)$ and $\mathbb V$ are orthogonal, 
\item
$Q_\mathbb G$ restricts to $Q_{\mathbb T}$ on $\Lie(\mathbb T)$, and 
\item
$Q_\mathbb G$ restricts to $Q_2$ on $\mathbb V$.
\end{itemize} 
It is clear that $Q_\mathbb G$ restricts to $Q_2$ on all of $\Lie(\mathbb
G_{\ssc})$.
It is also clear that $Q_\mathbb G$ is invariant under $\mathbb G_{\ad}$. Now $Q_2$
is invariant
under all automorphisms of $\mathbb G_{\ssc}$, and $Q_\mathbb T$ is invariant under
$\Omega_0$ (acting diagonally
on $\mathfrak a \times \mathfrak a$). 
It follows that $Q_\mathbb G$ is invariant under the natural action of $\mathbb
G_{\ad} \rtimes \Omega_0$ on
$\Lie(\mathbb G)$. (In greater detail, $\mathbb G_{\ad}$ is acting by the adjoint
representation,
and $\Omega_0$ is acting by 
the action induced by the diagonal action on $\mathbb G_{\ssc} \times \mathbb
T_{\ssc}$.)

In subsection \ref{sub.QVdef} a nondegenerate quadratic form 
$Q_{\mathbb V}$ is defined, and 
in  subsection \ref{sub.Q_VvsQ_2} there is a discussion of the
 orthogonal direct sum decomposition 
\begin{equation}\label{eq.qvqvqv}
(\mathbb V,Q_{\mathbb V}) = (\mathbb V',Q_{\mathbb V'}) \oplus (\mathbb V'',
Q_{\mathbb V''})
\end{equation}
as well as the relationship between the nondegenerate quadratic forms $Q_{\mathbb
V'}$, $Q_{\mathbb V''}$
and the ones we obtain 
by restriction from $Q_\mathbb G$. From this discussion we conclude that 
  $(\Lie(\mathbb G),Q_\mathbb G)$ decomposes as the  orthogonal direct sum  
\begin{equation}\label{eq.3summands}
(\Lie(\mathbb G),Q_\mathbb G) = (\Lie(\mathbb T),Q_\mathbb T) \oplus (\mathbb
V',Q_{\mathbb V'})
\oplus (\mathbb V'',\ell Q_{\mathbb V''}).  
\end{equation}
(The direct summand $\mathbb V''$ is zero in the simply laced case.)

\section{The split $F$-group $\mathbf G$ and the quadratic form $Q_\mathbf G$ on
its Lie algebra}

\subsection{The pinned $F$-group $\mathbf G$}
We are almost done with our preliminary analysis of the situation over $\mathbb Z$.
We extend scalars from $\mathbb Z$ to $F$. In this way we 
obtain a pinned $F$-group $(\mathbf G,\mathbf B,\mathbf T, \{\eta_\alpha\}) $.

\subsection{The quadratic forms $Q_\mathbf G$, $Q_{\mathbf T}$, 
$Q_\mathbf V$, $Q_{\mathbf V'}$, $Q_{\mathbf V''}$ } 
\label{sub.FourQuads}

By extension of scalars  the quadratic form $Q_{\mathbb G}$  yields 
 a quadratic form $Q_\mathbf G$ on $\mathbf g:=\Lie(\mathbf G)$. 
Similarly, we obtain from $Q_\mathbb T$ a quadratic form $Q_\mathbf T$ on $\mathbf
t:=\Lie(\mathbf T)$,
and from \eqref{eq.qvqvqv}  
we obtain $(\mathbf V,Q_{\mathbf V}) = (\mathbf V',Q_{\mathbf V'}) \oplus (\mathbf
V'',Q_{\mathbf V''})$.

The quadratic forms $Q_\mathbf V$, $Q_{\mathbf V'}$, 
$Q_{\mathbf V''}$ are always nondegenerate. 
We will observe later that, 
when $\ell$ is invertible in $F$, the quadratic forms 
$Q_\mathbf G$, $Q_{\mathbf T}$
are also nondegenerate. When $\ell$ is zero in $F$, 
we will need substitutes for $Q_\mathbf G$, $Q_{\mathbf T}$. 
We discuss these next. 

\subsection{The quadratic forms $Q_{\mathbf g'}$, $Q_{\mathbf g''}$, 
$Q_{\mathbf t'}$, $Q_{\mathbf t''}$} \label{sub.ManyQforms}

In this subsection we assume that $\ell$ is zero in $F$. Then $Q_\mathbf G$ is
degenerate and is therefore
not useful for proving the main theorem.   We get around this 
difficulty by 
applying Lemma \ref{lem.barQ} to $(\Lie(\mathbb G),Q_\mathbb G)$ and the three
direct summands
appearing in \eqref{eq.3summands}. 
For each of these four lattices $\Lambda$, Lemma \ref{lem.barQ} provides a short
exact sequence
\begin{equation}\label{eq.SESellLambdaPerp}
0 \to \Lambda^\perp/ \Lambda \xrightarrow{\ell} \Lambda/\ell\Lambda \to
\Lambda/\ell \Lambda^\perp \to 0
\end{equation}
as well as nondegenerate quadratic forms on the $\mathbb F_\ell$-vector spaces
$\Lambda/\ell \Lambda^\perp$
and $\Lambda^\perp/ \Lambda$. 

Now our assumption that $\ell$ is zero in $F$ means that $\mathbb F_\ell$ is the
prime field in $F$. So we may
tensor \eqref{eq.SESellLambdaPerp} over $\mathbb F_\ell$ with $F$, 
obtaining short exact sequences of $F$-vector spaces 
\[
0 \to \bar\Lambda'' \to \Lambda \otimes_{\mathbb F_\ell} F \to \bar\Lambda' \to 0,\]
as well as canonical nondegenerate quadratic forms on $\bar\Lambda'$ and
$\bar\Lambda''$.

We need notation for all these objects. When $\Lambda = \Lie(\mathbb G)$, we write\begin{equation}\label{SES.mathbfg}
0 \to \mathbf g'' \to \mathbf g \to \mathbf g' \to 0
\end{equation}
for the vector spaces and $Q_{\mathbf g''}$, $Q_{\mathbf g'}$ for the quadratic
forms.
Of course $\mathbf g$ is the Lie algebra of $\mathbf G$. 
Observe that \eqref{SES.mathbfg} is actually a short 
exact sequence of $\mathbf G_{\ad}$-modules, 
and that $Q_{\mathbf g''}$, $Q_{\mathbf g'}$ are both 
invariant under $\mathbf G_{\ad}$.

From the remaining three lattices $\Lie(\mathbb T)$, $\mathbb V'$, $\mathbb V''$ we
obtain short exact sequences of
$N_{\mathbf G}(\mathbf T)$-modules 
\begin{equation}\label{eq.3SES}
\begin{aligned}
&0 \to \mathbf t'' \to \mathbf t \to \mathbf t' \to 0 \\
&0 \to 0 \to \mathbf V' \xrightarrow{\id} \mathbf V' \to 0 \\
&0 \to \mathbf V'' \xrightarrow{\id} \mathbf V'' \to 0 \to 0 
\end{aligned}
\end{equation}
as well as quadratic forms $Q_{\mathbf t'}$, $Q_{\mathbf t''}$, $Q_{\mathbf V'}$,
$Q_{\mathbf V''}$.
It needs to be stressed here that the nondegenerate quadratic 
forms on $\mathbf V'$ and $\mathbf V''$ produced by 
Lemma \ref{lem.barQ} really do agree with the forms 
$Q_{\mathbf V'}$, $Q_{\mathbf V''}$ 
discussed in the previous subsection.  In other words, the factor of $\ell$ 
appearing in the last summand in \eqref{eq.3summands} has disappeared.

\subsection{Spinor norm on the Weyl group } 

We are done constructing all the quadratic forms we need, and are almost ready to
complete
the proof of the main theorem. Before doing so, we are going to 
calculate some spinor norms (see section \ref{sub.RevSpNorm} for a review of the
spinor norm map $\delta$)
 for elements in the Weyl group,  thereby paving the way for applying 
Fr\"ohlich's theorem. We begin 
by defining a sign character on $W$. 

\subsubsection{Definition of $\varepsilon''$}\label{sub.DefEp''}

On the Weyl group we have the usual sign character $\varepsilon:W \to \{\pm 1 \}$.
It takes the value
$-1$ on every reflection $w_\alpha = w_{\alpha^\vee} \in W$. More useful for us,
however, is another
sign character $\varepsilon''$ on $W$,  defined as follows. 

When $\ell =1$ (the simply laced case in which there is only one root length), we
take $\varepsilon''$ to be
$\varepsilon$. 
When $\ell \ne 1$, so that there are two root lengths, we take $\varepsilon''$ to
be
the unique sign character such that 
\[
\varepsilon''(w_\alpha) = 
\begin{cases}
1   & \text{ if $\alpha$ is long,} \\
-1  & \text{ if $\alpha$ is short. }
\end{cases}
\] 

We will mainly need $\varepsilon''$, but there will be one occasion when we will
also need the sign character
$\varepsilon'$ defined through the equality 
\[
\varepsilon =\varepsilon' \varepsilon''.
\]
We will need it only  when there are two root lengths, in which case  
\[
\varepsilon'(w_\alpha) = 
\begin{cases}
1   & \text{ if $\alpha$ is short,} \\
-1  & \text{ if $\alpha$ is long. }
\end{cases}
\]

We will sometimes view $\varepsilon'$ and $\epsilon''$ as having values in the
additive group $\mathbb F_2$, hence as
elements in $H^1(W,\mathbb F_2)$. 

\subsubsection{Spinor norms for $W$ when $\ell$ is invertible in $F$}

Assume that $\ell$ is invertible in $F$. We work with the split $F$-group $\mathbf
G$. We are
particularly interested in the 
nondegenerate quadratic space $(\mathbf t,Q_\mathbf T)$. The action of the Weyl
group preserves
$Q_\mathbf T$ and therefore yields a canonical homomorphism 
\begin{equation}\label{eq.WO}
W \to O(Q_\mathbf T)(F). 
\end{equation}
For any root $\alpha$ the coroot $H_\alpha$ is a vector in  
$\mathbf t$, and the homomorphism 
\eqref{eq.WO} maps $w_\alpha = w_{\alpha^\vee}$ 
to the reflection $r_{H_\alpha}$ in that vector.

Composing    the homomorphism 
\eqref{eq.WO}   with the spinor norm homomorphism 
\[
O(Q_\mathbf T)(F) \xrightarrow{\delta} F^\times/(F^\times)^2,
\]
we obtain a homomorphism 
\begin{equation}\label{eq.Sp1}
W \to F^\times/(F^\times)^2.
\end{equation}

\begin{lemma}\label{lem.SpinNormInvertible}
The value of the homomorphism \eqref{eq.Sp1} on $w \in W$ is equal to the square
class
of 
\[
\begin{cases}
1 & \text{ if $\varepsilon''(w)=1$ } \\
\ell  & \text{ if $\varepsilon''(w)=-1$ }
\end{cases}
\]
In other words  the homomorphism \eqref{eq.Sp1} is equal to the element 
\[
\varepsilon'' \otimes \ell \in H^1(W,\mathbb F_2) \otimes F^\times/(F^\times)^2 
= \Hom(W,F^\times/(F^\times)^2).
\]
\end{lemma}

\begin{proof}
It is enough to prove this when $w$ is a reflection $w_\alpha = w_{\alpha^\vee}$.
Then $w$ is sent to
 reflection in the coroot $H_\alpha$, whose spinor norm  
(see subsection \ref{sub.SpinNormRefl}) is the square 
class of $Q_\mathbf T(H_\alpha) = \ell(\alpha^\vee) \in F^\times$. 
\end{proof}

\subsubsection{Spinor norms for $W$ when $\ell$  is  zero in $F$} 

Now the Weyl group $W$ acts as a group of automorphisms of 
the nondegenerate quadratic spaces $(\mathbf t',Q_{\mathbf t'})$, $(\mathbf
t'',Q_{\mathbf t''})$, so
there are natural homomorphisms from $W$ to $O(Q_{\mathbf t'})(F)$ and
$O(Q_{\mathbf t''})(F)$.
As usual we denote spinor norm maps by $\delta$. 

\begin{lemma}\label{lem.SpinNormZero}
The composed maps 
\begin{align}
W \to O(Q_{\mathbf t'})(F) \xrightarrow{\delta}  F^\times/(F^\times)^2 \\
W \to O(Q_{\mathbf t''})(F) \xrightarrow{\delta}  F^\times/(F^\times)^2
\end{align}
 are both trivial. 
\end{lemma}

\begin{proof} 
Let $w \in W$ and write $w'$ (resp., $w''$) for the image of $w$ in $O(Q_{\mathbf
t'})$
(resp., $O(Q_{\mathbf t''})$). 
It is enough to prove that the spinor norms of $w'$ and $w''$ are trivial when $w$
is a reflection. So let us
consider the reflection $w = w_\alpha = w_{\alpha^\vee} \in W$. 
From Corollary \ref{cor.QTalpha} we know that 
\begin{itemize} 
\item 
$Q_{\mathbb T}(\alpha^\vee) = 1$ when $\alpha$ is long, and 
\item 
$(\ell Q_{\mathbb T})(\alpha) = 1$ when $\alpha$ is short. 
\end{itemize}
So, when $\alpha$ is long (resp., short), Lemma \ref{lem.Qv1} implies that
$(w',w'')$ is equal to
$(r_{v'},1)$ (resp., $(1,r_{v''})$), where $v'$ is the image of $\alpha^\vee$ in
$\Lambda/\ell\Lambda^\perp$
(resp., $v''$ is the image of $\alpha$ in $\Lambda^\perp/\Lambda$). Now 
$Q_{\mathbf t'}(v')=Q_\mathbb T(\alpha^\vee) =1$ 
(resp., $Q_{\mathbf t''}(v'')=(\ell Q_\mathbb T)(\alpha) =1$), so 
(see subsection \ref{sub.SpinNormRefl}) all the relevant spinor 
norms are indeed trivial. 
\end{proof}

\section{End of the proof of the main theorem}

\subsection{The quasi-split $F$-group $G$ and the quadratic form $Q_G$} 

The Galois group acts 
trivially on $\Omega_0$, so $1$-cocycles of $\Gamma$ in $\Omega_0$ are simply 
homomorphisms $\varphi_0 : \Gamma \to \Omega_0$.  
As before we are interested in   
the diagonal action of $\Omega_0$ on $\mathbf G_{\ssc} \times \mathbf T_{\ssc}$ and
the induced action on
the quotient $\mathbf G$. We use this 
action, together with $\varphi_0$, to twist $(\mathbf G,\mathbf B,\mathbf T,
\{\eta_\alpha\}) $.
The result is a quasi-split $F$-group $G$ equipped with an $F$-splitting 
 $(B_0, T_0,\{\eta_\alpha\} )$. 
(One has $\sigma(\eta_\alpha) = \eta_{\sigma\alpha}$ for all 
$\sigma \in \Gamma$.) We are writing $T_0$ for the twist of 
$\mathbf T$ by $\varphi_0$ in order to keep $T$ in reserve as notation 
for an arbitrary maximal $F$-torus in $G$. The $F$-torus $T_0$ 
is of course very special, because there exists
a Borel $F$-subgroup, namely $B_0$, containing it. 

Now the action of $\Omega_0$ preserves the quadratic 
form $Q_{\mathbf G}$, which therefore  remains $F$-rational  
after twisting.  In other words 
we may now view it as a canonically defined quadratic form, 
call it $Q_G$, on $\mathfrak g =\Lie(G)$.
See subsection \ref{sub.TwistQ}  for a review of twisting of 
quadratic forms. 

It follows from Lemmas \ref {lem.QTint} and \ref{lem.barQ}, 
together with our comparison of $Q_2$
and $Q_\mathbb V$ (see subsection \ref{sub.Q_2})   
that $Q_G$ is nondegenerate when $\ell$ is invertible in $F$.

\subsection{What remains to be proved?}

For any two maximal $F$-tori $T$, $\tilde T$ in $G$ we need to prove that 
\[
\gamma(Q_{\tilde V},\psi)\gamma(Q_V,\psi) ^{-1} = \epsilon_L(X_*(\tilde T)_\mathbb
C - X_*(T)_\mathbb C,\psi).
\]
Applying Corollary \ref{cor.Wep}, we are reduced to proving that 
\begin{equation}\label{eq.SWgoal}
HW(Q_V,Q_{\tilde V}) = \overline{SW}(X_*(\tilde T) - X_*(T)) . 
\end{equation}
(The expressions $HW$ and $\overline{SW}$ appearing here are discussed in the
appendices.)

The maximal torus $T$ is obtained from $T_0$ by twisting by a $1$-cocycle in the
normalizer of $T_0$
in $G$. So $T$ is obtained from $\mathbf T$ by twisting by a $1$-cocycle of the
form $n_\sigma\varphi_0(\sigma)$
with $n_\sigma$ in the $\overline{F}$-points of the normalizer $N_{\mathbf
G}(\mathbf T)$ of $\mathbf T$ in $\mathbf G$.
The same goes for $\tilde T$, which is obtained from $\mathbf T$ by twisting by 
a $1$-cocycle of the form $\tilde n_\sigma\varphi_0(\sigma)$. 

At this point we need to consider two cases: either $\ell$ is invertible in $F$, or
$\ell$ is zero in $F$.

\subsection{End of the proof when $\ell$ is invertible in $F$}
We now assume that $\ell$ is invertible in $F$, which ensures that $Q_G$ is
nondegenerate.
We write $Q_T$ for the restriction of $Q_G$ to $\mathfrak t =\Lie(T)$; it is
obtained by twisting $Q_{\mathbf T}$.
 In fact the whole situation 
$(\mathfrak g, Q_G) = (\mathfrak t, Q_T) \oplus (V, Q_G|_V)$ 
is obtained by twisting 
$(\mathbf g,Q_\mathbf G) = (\mathbf t,Q_\mathbf T) \oplus (\mathbf V,Q_\mathbf
G|_{\mathbf V})$
by the $1$-cocycle $n_\sigma\varphi_0(\sigma)$. (Recall from subsection
\ref{sub.FourQuads}
that  $\mathbf V$ is 
obtained by extension of scalars from  $\mathbb V$.) 

We let $w_\sigma$ denote the image of $n_\sigma$ under the natural homomorphism
$N_{\mathbf G}(\mathbf T) \twoheadrightarrow W$. Now put
$\varphi(\sigma):=w_\sigma\varphi_0(\sigma) \in W \rtimes \Omega_0 =
\Omega$. Then $\varphi$ is a homomorphism  $\Gamma \to \Omega$, and 
it is also true that $T$ is obtained from the split torus 
$\mathbf T$ by twisting by $\varphi$, viewed as a $1$-cocycle of $\Gamma$ 
in $\Omega$, with $\Gamma$ acting trivially on $\Omega$. Similarly we obtain
$\tilde w_\sigma$, $\tilde \varphi$
from $\tilde T$. In addition to $\varphi,\tilde \varphi$ we will soon need the sign
character $\chi$ on
$\Gamma$ defined by  
\begin{equation}\label{eq.DefChi}
\chi(\sigma) := \varepsilon''(\tilde w_\sigma w_\sigma^{-1}) =
 \varepsilon''( \tilde{\varphi}(\sigma) \varphi(\sigma)^{-1}).
\end{equation}
To see that $\chi$ really is multiplicative in $\sigma$, one needs to notice that
the sign
character $\varepsilon'' $ 
(defined in  subsection \ref{sub.DefEp''}) is invariant 
under  the action of $\Omega_0$ on $W$.

Twisting the decomposition 
$\mathbf V = \mathbf V' \oplus \mathbf V''$, we obtain a decomposition  
$V  = V' \oplus V''$. When $\ell =1$ (the simply laced case)   
$V'=V$ and $V''=0$. 
When $\ell = 2,3$ (the non-simply laced case) 
$V'$ is the sum of the root spaces for the long roots of $T$,  and $V''$ 
is the sum of the root spaces for the short roots of $T$.  
The twist of $Q_\mathbf V$ coincides with the quadratic form $Q_V$ 
defined in subsection \ref{subsub.ThirdIngredient}. 
We denote the restriction of $Q_V$ to $V'$ (resp., $V''$) by $Q_{V'}$ (resp.,
$Q_{V''}$).

It follows from \eqref{eq.3summands} that 
\begin{equation} \label{eq.TV'V''}
(\mathfrak g, Q_G) = (\mathfrak t, Q_T) \oplus (V', Q_{V'}) 
\oplus (V'',\ell Q_{V''}). 
\end{equation}
All these quadratic spaces are nondegenerate.  
 From this orthogonal decomposition  
(and its  analog for $\tilde T$), we find that 
\[
\Wall(Q_T) + \Wall(Q_{V'}) + \Wall(\ell Q_{V''})  
\]
is equal to its analog for $\tilde T$
(see Appendix \ref{app.WgWh} for a discussion of the Wall 
homomorphism). 
It follows  from Lemma \ref{lem.HWvs Wall}(2) that 
\begin{align*}
0 &= HW(Q_T,Q_{\tilde T}) + HW(Q_{V'}, Q_{\tilde V'}) + HW(\ell Q_{V''}, \ell
Q_{\tilde V''}) \\
&= HW(Q_T,Q_{\tilde T}) + HW(Q_{V}, Q_{\tilde V}) + HW(\ell Q_{V''}, \ell Q_{\tilde
V''})
- HW( Q_{V''}, Q_{\tilde V''}).
\end{align*}

To finish the proof of \eqref{eq.SWgoal}, it remains only to justify the following
claims.
\begin{enumerate}
\item
$HW(Q_T,Q_{\tilde T}) = \overline{SW}(\varphi) - \overline{SW}(\tilde \varphi) +
\xi(\chi\otimes \ell)$.
\item
$HW(\ell Q_{V''},\ell Q_{\tilde V''}) = HW(Q_{V''},Q_{\tilde V''}) + \xi(\chi
\otimes \ell)$.
\item $\overline{SW}(X_*(T)_\mathbb C) = \overline{SW} (\varphi)$, 
and similarly for $(\tilde T,\tilde\varphi)$.
\end{enumerate}
Here $\chi$ is the sign character on $W$ we defined (see equation  \eqref{eq.DefChi}) using
$\varepsilon''$,
$\varphi$ and $\tilde \varphi$.  The homomorphism $\xi$ is 
defined in Appendix \ref{app.FrohArb}.

For item (1) we appeal to the result of Fr\"ohlich which is reviewed (and extended
to characteristic $2$) in subsection \ref{sub.FrohReview}.
We apply his result to  $\varphi$, viewed as an 
 orthogonal representation of $\Gamma$ on 
the quadratic space $(\mathbf t,Q_{\mathbf T})$ over $F$. Fr\"ohlich's result says
that
\[
HW(Q_T,Q_\mathbf T) = \overline{SW}(\varphi) + \xi(\delta \circ \varphi), 
\] 
and similarly with $(T,\varphi)$ replaced by $(\tilde T,\tilde{\varphi})$. 
It follows that 
\begin{align*}
HW(Q_T,Q_{\tilde T}) &= HW(Q_T,Q_\mathbf T) + HW(Q_\mathbf T,Q_{\tilde T}) \\
&= \overline{SW}(\varphi) + \xi(\delta \circ \varphi) -
\overline{SW}(\tilde{\varphi}) - \xi(\delta \circ \tilde{\varphi})\\
&= \overline{SW}(\varphi)  -\overline{SW}(\tilde\varphi)  + \xi(\chi \otimes \ell).\end{align*}
The last equality follows from   Lemma \ref{lem.SpinNormInvertible}. 

For item (2) we apply Lemma \ref{cor.WaSC}, bearing in mind 
equation \eqref{eq.XiVsLocSymb}, thereby obtaining  
\begin{align*}
\Wall(\ell Q_{V''}) &= \Wall( Q_{V''}) + \xi(\chi_{Q_{V''}} \otimes \ell), \\ 
\Wall(\ell Q_{\tilde V''}) &= \Wall( Q_{\tilde V''}) + \xi(\chi_{Q_{\tilde V''}}
\otimes \ell),
\end{align*} 
from which it follows that 
\[
HW(\ell Q_{V''},  \ell Q_{\tilde V''} )      = HW( Q_{V''}, Q_{\tilde V''}  ) + 
\xi(\chi_{Q_{V''}}^{-1}   \chi_{Q_{\tilde V''}} \otimes \ell)  
\]
So we just need to show that $\chi_{Q_{\tilde V''}}$ is the product of $\chi_{Q_{
V''}}$ and $\chi$.

For this it is enough to prove that the following square commutes: 
\begin{equation*}
\begin{CD}
N_\mathbb G(\mathbb T) @>>> O(Q_{\mathbb V''}) \\
@VVV @V{\deg}VV \\
W @>{\varepsilon''}>> \mathbb F_2.
\end{CD}
\end{equation*}

Let's call the top horizontal arrow $t$. It is clear that  the composed map 
$\deg \circ \, t$ is trivial on the identity component $\mathbb T$, so we are
really trying
to prove the equality of two sign characters on $W$. For this it suffices to show
that they
agree on every simple reflection. So, let  $\alpha$ be  a simple root, and 
consider  $n_\alpha \in N_\mathbb G(\mathbb T)$ whose image in $W$ is the simple 
reflection $s_\alpha$. We need to show that $\deg(t(n_\alpha))$ is trivial (resp.,
nontrivial)
if $\alpha$ is long (resp., short). 

Recall that $\mathbb V''$ is the direct sum of the root spaces for all of the short
roots.
Now the set $R''$ of short roots is the disjoint union of the sets $R'' \cap R^+$
and $R'' \cap R^-$,
so  $\mathbb V''$ decomposes accordingly as 
\begin{equation}\label{eq.V''pm}
\mathbb V''= \mathbb V''^+ \oplus \mathbb V''^-. 
\end{equation}
The action of $s_\alpha$ on $R$ interchanges $\alpha$, $-\alpha$.
It also preserves 
the sets $R^+ \setminus \{\alpha\}$ and $R^- \setminus \{ -\alpha \}$, as well as
the set $R''$.

So, when $\alpha$ is long, $s_\alpha$ preserves both $R'' \cap R^+$ and $R'' \cap
R^-$, and therefore the orthogonal transformation
$t(n_\alpha)$ preserves the  decomposition \eqref{eq.V''pm}
of $\mathbb V''$ as the orthogonal direct sum of two isotropic 
subspaces. It follows that $\deg(t(n_\alpha))$ is trivial, as 
desired. 

When $\alpha$ is short, similar considerations imply that 
$\deg(t(n_\alpha)) = \deg(t_{\pm \alpha}(n_\alpha))$,
where $ t_{\pm \alpha}(n_\alpha)$ denotes the image of $n_\alpha$ in the orthogonal
group of the hyperbolic plane
$\Lie(\mathbb G_\alpha) \oplus \Lie(\mathbb G_{-\alpha})$.
But $n_\alpha$ interchanges these two root spaces, so the sign 
character $\deg$ for the orthogonal group of this hyperbolic plane 
is nontrivial on $t_{\pm \alpha}(n_\alpha)$, as desired.
The verification of item (2) is now complete.

Finally, for item (3)  we  appeal to Appendix \ref{app.SWredEll}.  
We are now  done with the case in which $\ell$ is invertible in $F$.

\subsection{The case when $\ell$ is zero in $F$}

Now assume that $\ell$ is $0$ in $F$, which is to say that the prime field of $F$
is $\mathbb F_\ell$.
The canonical quadratic form $Q_G$ is then degenerate, so it does not have a Weil
index, and we must modify
the procedure used when $\ell$ is invertible in $F$. 

Since the hypothesis that $\ell = 0$ in $F$ implies that $\ell$ is either 
$2$ or $3$, the irreducible root system $R$ under consideration is not simply
laced, and the group $\Omega_0$ is
necessarily trivial. So no twisting takes place, and $G$ coincides with the split
$F$-group $\mathbf G$
obtained from $\mathbb G$ by extension of scalars from $\mathbb Z$ 
to $F$. Moreover $T_0$, $\Lie(T_0)$, $\mathfrak g$ coincide with 
$\mathbf T$, $\mathbf t$, $\mathbf g$ respectively.

The maximal torus $T$ is obtained from $\mathbf T$ 
by twisting by a $1$-cocycle $n_\sigma$ in
$N_\mathbf G(\mathbf T)$.   As before we denote by 
$w_\sigma$ the image of $n_\sigma$ under 
$N_\mathbf G(\mathbf T) \to W$. Because $\mathbf T$ is split,
the map $\varphi: \Gamma \to W$ 
defined by $\varphi(\sigma) = w_\sigma$ is a homomorphism.

Twisting the short exact sequences \eqref{eq.3SES} by the $1$-cocycle
$n_\sigma$   then yields short exact sequences 
\begin{equation}
\begin{aligned}
&0 \to \mathfrak t'' \to \mathfrak t \to \mathfrak t' \to 0 \\
&0 \to 0 \to V' \xrightarrow{\id} V' \to 0 \\
&0 \to V'' \xrightarrow{\id}  V'' \to 0 \to 0 
\end{aligned}
\end{equation}
as well as quadratic forms $Q_{\mathfrak t'}$, $Q_{\mathfrak t''}$, $Q_{ V'}$, $Q_{
V''}$.
Moreover $Q_{V'}$ and $Q_{V''}$ coincide with the quadratic forms  
obtained by restriction from $Q_V$ on $V=V'\oplus V''$ (see the remark 
at the end of subsection \ref{sub.ManyQforms}).

Observe that there are $T$-invariant orthogonal direct sum decompositions 
\[
\mathfrak g' = \mathfrak t' \oplus V', \quad \mathfrak g'' = \mathfrak t'' \oplus
V''.
\]
To finish the proof we are going to use $(\mathfrak g' \oplus \mathfrak g'',
Q_{\mathfrak g'} \oplus
Q_{\mathfrak g''})$ in the same way we used $(\mathfrak g,Q_G)$ before. Observe
that
\begin{equation}\label{eq.Qg'g''frak}
Q_{\mathfrak g'} \oplus Q_{\mathfrak g''} = Q_{\mathfrak t'} \oplus Q_{\mathfrak
t''} \oplus Q_V.
\end{equation}
So, although the notation is more elaborate, the situation is simpler in one
respect: the
restriction of $Q_{\mathfrak g'} \oplus Q_{\mathfrak g''}$ to $V' \oplus V'' = V$
is precisely
the canonical quadratic form $Q_V$ we need to understand.

Everything we have just done for $T$ can also be done for $\tilde T$.  
Objects for $\tilde T$ are always denoted by adding a tilde 
to the name of the corresponding object for $T$.
From \eqref{eq.Qg'g''frak} and  
 its analog for $\tilde T$ we find that 
\[
0 = HW(Q_{\mathfrak t'},Q_{\tilde{\mathfrak t}'}) + HW(Q_{\mathfrak
t''},Q_{\tilde{\mathfrak t}''})
+ HW(Q_V,Q_{\tilde V}). 
\]
To complete the proof of \eqref{eq.SWgoal} it suffices to verify the following
claims.
\begin{enumerate}
\item
$HW(Q_{\mathfrak t'},Q_{\tilde{\mathfrak t}'}) = \overline{SW}(\varphi') -
\overline{SW}(\tilde{\varphi}')$.
\item 
$HW(Q_{\mathfrak t''},Q_{\tilde{\mathfrak t}''}) = \overline{SW}(\varphi'')
-\overline{SW}(\tilde{\varphi}'')$.
\item 
$\overline{SW}(X_*(T)_\mathbb C) = \overline{SW}(\varphi') +
\overline{SW}({\varphi}'')$.
\end{enumerate}
Here $\varphi'$ (resp., $\varphi''$) denotes the orthogonal representation of
$\Gamma$ on $\mathbf t'$ (resp.,
$\mathbf t''$) 
obtained from $\varphi:\Gamma \to W$ and the homomorphism 
$W \to O(Q_{\mathbf t'})$ 
(resp.,  $W \to O(Q_{\mathbf t''})$).

To prove  (1) we observe that 
\begin{align*}
HW(Q_{\mathfrak t'},Q_{\tilde{\mathfrak t}'}) &= 
HW(Q_{\mathfrak t'},Q_{\mathbf t'}) - 
HW(Q_{\tilde{\mathfrak t}'},Q_{\mathbf t'}) \\
&= \overline{SW}(\varphi') - \overline{SW}(\tilde{\varphi}').
\end{align*}
Here we applied  Fr\"ohlich's result 
 to $\varphi'$, $\tilde{\varphi}'$, using Lemma \ref{lem.SpinNormZero},  
 which asserts  
that the spinor norm map is trivial on the image of $W$ in 
$O(Q_{\mathbf t'})$. Claim (2) is proved in the same way. 

To prove (3) we need to consider the subcases $\ell = 2$ 
and $\ell = 3$ separately. When $\ell = 2$, our field $F$ 
has characteristic $2$, and the only real content in (3) is
that $\overline{SW}_1(X_*(T)_\mathbb C) = \overline{SW}_1(\varphi') +
\overline{SW}_1({\varphi}'')$.
This is not quite as obvious as one might think, 
since the sign character $\deg$ on orthogonal groups
in characteristic $2$ is not given by the determinant. 
Fortunately the Weyl group is generated by
reflections, so Lemma \ref{lem.Qv1} does the job. 
(In characteristic $2$ it remains true that 
the sign character $\deg$ is nontrivial on reflections.)

It remains to prove (3) when $\ell = 3$. We have three orthogonal representations
of $W$,
one on the $\mathbb C$-vector space $X_*(\mathbb T)_\mathbb C$,  one on 
the $F$-vector space $\mathbf t'$, and one on  the $F$-vector space $\mathbf t''$. It suffices to prove the equality 
\[
{SW}(X_*(\mathbb T)_\mathbb C)      = {SW}(\mathbf t') + {SW}(\mathbf t'') 
\]
of elements in  the group $1 + H^1(W,\mathbb F_2) + H^2(W,\mathbb F_2)$.

Since $\ell=3$, our root system is $G_2$. Therefore $\mathbb G_{\ssc} = \mathbb
G_{\ad}$,
and $\mathbb T$ is simply the cartesian product of two copies of $\mathbb
T_{\ssc}$. The second copy
(the one that is central in $\mathbb G$) has 
no effect on anything, so we are reduced to proving 
the equality   
\[
{SW}(X_*(\mathbb T_{\ssc})_\mathbb C)      = {SW}(\mathbf u') + {SW}(\mathbf u'') 
\in 1 + H^1(W,\mathbb F_2) + H^2(W,\mathbb F_2) ,
\]
where $\mathbf u''$ is the line in 
$\mathbf u := X_*(\mathbb T_{\ssc})_F$ generated by any 
 long coroot, and $\mathbf u'$ is the quotient line 
 $\mathbf u / \mathbf u''$. 
(Remember that $F$ has characteristic $3$. The reflection 
representation $\mathbf u$ of $W$ is 
reducible, but not semisimple. The unique $W$-invariant line 
is $\mathbf u''$; it contains all the long coroots.)

Recall from subsection \ref{sub.DefEp''} that the sign character 
$\varepsilon$ on $W$ is the product 
of two sign characters $\varepsilon'$ and $\varepsilon''$. 
One checks that $W$ acts on the line $\mathbf u'$ (resp., $\mathbf u''$) 
by the sign character $\varepsilon'$ (resp., $\varepsilon''$).

The Weyl group is dihedral of order $12$. It is not difficult to see that 
\[
{SW}(X_*(\mathbb T_{\ssc})_\mathbb C) = 
1 + \varepsilon + \varepsilon' \cup \varepsilon'' \in 1 + H^1(W,\mathbb F_2) +
H^2(W,\mathbb F_2),
\] 
and this is the sum of 
$SW(\mathbf u') = 1 + \varepsilon'$ and $SW(\mathbf u'') = 1 + \varepsilon''$, as
desired.
The group law used here is the usual one:
\[
(1 + a + b)(1+ a' +b') = 1 + (a+a') + (a \cup a' + b + b'). 
\]

\appendix

\section{Witt group and the Wall homomorphism} \label{app.WgWh}

In this section we work over an arbitrary field $F$. By algebra we mean
$F$-algebra, and by dimension we mean $F$-dimension.  We choose a
separable closure $\overline F$ of $F$ and write $\Gamma$ for the Galois
group $\Gal(\overline F/F)$. Throughout Appendix \ref{app.WgWh} all quadratic
spaces
we consider 
are tacitly assumed to be nondegenerate, as these are the only ones relevant for
the Witt group.

\subsection{Quadratic \'etale algebras} Recall that an \emph{\'etale}
algebra is product of finitely many finite separable field extensions of
$F$. A \emph{quadratic} \'etale algebra is an \'etale algebra of dimension
$2$.  A quadratic \'etale algebra $E$ is either a separable quadratic field
extension or else isomorphic to $F \times F$. In either case there is a
unique nontrivial automorphism of $E$ over $F$, denoted by $y \mapsto \bar
y$. There is a canonical quadratic form $Q_E$ on the $2$-dimensional
$F$-vector space $E$, given by
$Q_E(y)=y\bar y$. Of course $Q_E$ is nothing but the norm mapping
$N_{E/F}:E \to F$.

Attached to $E$ is a canonical sign character
$\chi_{E}$ on $\Gamma$. When $E$ is isomorphic to $F \times F$, the
character $\chi_{E}$ is trivial. When $E/F$ is a separable quadratic
field extension, $\chi_{E}$ is the composed map 
\[
\Gamma \twoheadrightarrow \Gal(E/F) \simeq \{\pm 1 \} 
\] 
(which is of course independent of the choice of embedding of $E$ in
$\overline F$). 

\subsection{Quaternion algebras} A \emph{quaternion} algebra $D$  
is a central simple algebra of dimension $4$. There are two possibilities.
One is that $D$ is a division algebra. The other is that $D$ is
\emph{split}, i.e.~isomorphic to $M_2F$. In any case there is a canonical
anti-involution $y \mapsto \bar y$ on $D$. It is characterized by the
fact that $y + \bar y$ (resp., $y\bar y$) is the reduced trace (resp.,
reduced norm) of $y$. There is a canonical quadratic form $Q_D$ on the
$4$-dimensional vector space $D$, 
given by $Q_D(y)=y\bar y = \bar y y$. 

\subsection{The quaternion algebras $D(E,a)$}
There is a standard way of  attaching a $\mathbb
Z/2\mathbb Z$-graded quaternion algebra
$D=D(E,a)$ to a pair consisting of a quadratic \'etale algebra $E$ and an
element $a \in F^\times$. From now on we will write $\mathbb F_2$ rather
than $\mathbb Z/2\mathbb Z$.  The isomorphism class of
$D(E,a)$ (as
$\mathbb F_2$-graded  algebra) depends only on $a$ modulo norms
from $E^\times$. 

Here is a description of $D=D_0 \oplus D_1$. As before we write $e \mapsto
\bar e$ for the unique nontrivial automorphism of $E$ over $F$. 
\begin{enumerate}
\item The even part $D_0$ of $D$ is $E$. So $E$ is a subalgebra of $D$. 
\item The odd part  $D_1$ of $D$ is  
free of rank $1$ as both left  and right $D_0$-module, hence is  
$2$-dimensional over $F$. 
\item There is an element $x \in D_1$ such that 
\begin{enumerate}
\item $x^2=a \in F \subset
D_0$, 
\item   $xe=\bar ex$  for all $e \in E = D_0$. 
\end{enumerate}
\end{enumerate} 
 Observe that $x$ is  a basis for $D_1$ as both left
and right $D_0$-module. 

 The canonical
anti-involution on $D$ preserves the $\mathbb F_2$-grading. It 
\begin{enumerate}
\item acts on $D_0=E$ by the unique nontrivial 
automorphism of $E/F$, and 
\item acts on $D_1=Ex$ by multiplication by $-1$. 
\end{enumerate}
Because of the first item, there is no conflict  
in using $y \mapsto \bar y$
to denote both  the unique nontrivial 
automorphism of $E/F$ and the canonical anti-involution of $D/F$. 

\subsection{Quadratic spaces}

A quadratic space is a pair $(V,Q)$ consisting of a finite dimensional
vector space $V$ and a quadratic form $Q$ on $V$ (i.e.~$Q$ is an
element in the second symmetric power of $V^*$). 
We then obtain a symmetric bilinear form $B$ on $V$
defined by
$Q(x+y)=Q(x)+B(x,y)+Q(y)$. Note that
$B(x,x)=2Q(x)$. Thus $B$ is an alternating form when $F$ has
characteristic $2$. A quadratic space $(V,Q)$ is said to be
nondegenerate when the bilinear form $B$ is nondegenerate. 
 In characteristic $2$ a nondegenerate quadratic space is 
 necessarily even dimensional. 
Throughout  this appendix we tacitly assume that the quadratic 
spaces we consider are nondegenerate.

\subsection{Even Witt group $W_0(F)$} 
We write $W_0(F)$ for the Witt group of even-dimensional quadratic
spaces  over $F$, and we call $W_0(F)$ the \emph{even Witt group} of $F$. 
Given an even dimensional quadratic space
$(V,Q)$ over
$F$, we  write $[Q]$ for the class of $(V,Q)$ in the even Witt group. We
remind the reader that the equations
\begin{align*}
[Q]+[-Q]&=0, \\
[Q_{F\times F}]&=0 
\end{align*}
hold in $W_0(F)$. (Notice that $Q_{F\times F}$ 
is a hyperbolic plane.) 

Given an even dimensional quadratic space $(V,Q)$ over $F$, 
we write $C(Q)$ for the Clifford algebra of $Q$. 
Then $C(Q)$ is a central simple algebra of
dimension $2^n$, where $n=\dim(V)$. Moreover $C(Q)$ is an 
$\mathbb F_2$-graded algebra. It is a standard fact that 
\begin{equation}\label{eq.PlusTimes}
C(Q_1 \oplus Q_2)=C(Q_1) \otimes_s C(Q_2). 
\end{equation}
Here $Q_1 \oplus Q_2$ denotes the
orthogonal direct sum of $Q_1$ and $Q_2$, 
and $\otimes_s$ denotes the
$\mathbb F_2$-graded tensor product of 
$\mathbb F_2$-graded algebras.  

Here are some standard examples. The isomorphisms in the
proposition are isomorphisms of $\mathbb F_2$-graded algebras. 

\begin{proposition} \label{prop.ClEx}
\hfill 
\begin{enumerate}
\item Let $E$ be a quadratic \'etale algebra and let $a \in F^\times$. 
Then $C(aQ_E)$ is isomorphic to $D(E,a)$. 
\item Let $D$ be any quaternion algebra. Then $C(Q_D)$ 
is isomorphic to
the algebra $M_2(D)$ of $2\times2$-matrices 
with entries in $D$, graded
by taking the diagonal matrices as the even part, 
and the anti-diagonal matrices as the odd part.  
\end{enumerate}
\end{proposition} 
\begin{proof} In this proof 
we will use the following  way of 
recognizing the Clifford algebra $C(Q)$ obtained from an even-dimensional 
quadratic space
$(V,Q)$.   Let
$A=A_0 \oplus A_1$ be an $\mathbb F_2$-graded algebra having the same
dimension as $C(Q)$. Then giving an isomorphism $\rho:C(Q) \to A$ of
$\mathbb F_2$-graded algebras is the same as giving a linear map $\phi:V
\to A_1$ such that $(\phi(v))^2=Q(v)$ for all $v \in V$. 
(The simplicity of the algebra $C(Q)$ guarantees  
that $\rho$ is injective, hence an isomorphism.) 

First we prove (1). We abbreviate $D(E,a)$ to $D$. Giving an isomorphism
$\rho:C(aQ_E)\to D$ is the same as giving a linear map $\phi:E \to D_1$
such that $\phi(e)^2=ae\bar e$.  Taking 
$\phi:E \to D_1$ to be 
\[
e  \mapsto  ex \in Ex=D_1 
\] 
does the job. 
Here we are using  the description of  $D(E,a)$ given earlier. 

Now we prove (2).   
Identify $M_2(D)$ with
the $\mathbb F_2$-graded endomorphism algebra $\End_D(M)$ of
the $\mathbb F_2$-graded right $D$-module $M=D \oplus D$, the first copy
of $D$ being given even degree, and the second copy being given odd
degree. 
Giving $\rho$ is then the same as giving  
 an $F$-linear map $d \mapsto \phi_d$
from $D$ to the odd part of $\End_D(M)$, subject to the condition 
$(\phi_d)^2=d\bar d$.   Taking 
\[
\phi_d(x,y):=(\bar d y, dx) \qquad (x,y \in D)
\] 
does the job. 
\end{proof} 

A related fact is 

\begin{proposition}
Let $E$ be a quadratic \'etale algebra and let $a \in F^\times$. 
Put $D=D(E,a)$.
Then $Q_D$ is the orthogonal direct sum of $Q_E$ and 
$-aQ_E$. Consequently there is an equality 
\[
[Q_D]=[Q_E]-[aQ_E]
\] 
in the even Witt group. 
\end{proposition}

\begin{proof} 
As before we have $D=D_0 \oplus D_1=E\oplus Ex$. 
It is easy to see that the direct sum decomposition $D=D_0 \oplus D_1$ is
orthogonal, and that $D_0$ (resp., $D_1$) is
isomorphic as quadratic space to $(E,Q_E)$ (resp., $(E,-aQ_E)$). 
\end{proof} 

\subsection{Wall's little graded Brauer group} 

The reference for this subsection is Wall's paper \cite{Wall}. 
Wall's little graded Brauer
group 
 is an analog of the usual Brauer group $\Br(F)$. It is
based on the $\mathbb F_2$-graded tensor product of $\mathbb F_2$-graded
central simple algebras, just as the usual Brauer group is based on the
ordinary tensor product of central simple algebras. The word ``little''
refers to the fact that one only considers $\mathbb F_2$-graded algebras
whose underlying algebra is central simple in the usual sense.

Wall has determined the structure of the little graded Brauer group, which
we denote by
$\Br_s(F)$. (We consistently use a subscript $s$ to indicate ``super'' analogs of
standard concepts.)
He did this by attaching two invariants to an 
 $\mathbb F_2$-graded central simple algebra $A=A_0 \oplus A_1$. 
 The first
is the sign character $\chi_{E}$ associated to the quadratic \'etale
algebra $E=E(A)$ defined as the center of $A_0$. The second is the
class $[A]$ of $A$ in the ordinary Brauer group $\Br(F)$. We will use
additive notation for $\Br(F)=H^2(F,\overline F^\times)$. Likewise, 
we use additive notation in
$H^1(F,\mathbb F_2)$, which we identify with the group of (continuous) sign
characters on $\Gamma$. 
\begin{proposition}[Wall] 
\hfill 
\begin{enumerate}
\item The map
\[
A \mapsto (\chi_{E(A)},[A])
\]
induces a bijection from $\Br_s(F)$ to the Cartesian product 
\[
H^1(F,\mathbb F_2) \times \Br(F). 
\]
\item Use the isomorphism in the first part to put the structure of abelian
group on $H^1(F,\mathbb
F_2) \times \Br(F)$. The addition law is then
given by 
\[
(\chi,x)+(\chi',x')=(\chi+\chi',x+x'+\chi\cup\chi'), 
\] 
where $\cup$ is the cup-product pairing 
\[
H^1(F,\mathbb F_2) \otimes
H^1(F,\mathbb F_2) \to H^2(F,\bar F^\times)
\] 
induced by the pairing
$\mathbb F_2 \otimes \mathbb F_2 \to \overline F^\times$ given by 
$m \otimes n \mapsto (-1_F)^{mn}$. 
\end{enumerate}
\end{proposition}

Observe that the pairings occurring in this proposition are trivial when
the characteristic of $F$ is $2$. In this case, the little graded
Brauer group is isomorphic as abelian group to the cartesian product 
$H^1(F,\mathbb F_2) \times \Br(F)$.

We use the bijection in the proposition to identify $\Br_s(F)$ with 
$H^1(F,\mathbb
F_2) \times \Br(F)$ set-theoretically. Thus, given an $\mathbb F_2$-graded
central simple algebra $A$, we view the class $[A]_s$ of $A$ in $\Br_s(F)$
as being the pair $(\chi_{E(A)},[A])$. 

Notice that the cup-product $\chi\cup\chi'$
occurring in the formula for  the addition law always  lies in the 
subgroup $\Br_2(F):=\{ x \in \Br(F): 2x=0 \}$ of the Brauer group. 
Therefore $\Br_2(F)_s:=\{(\chi,x) \in H^1(F,\mathbb F_2) \times
\Br(F): 2x=0 \}$ is a subgroup of the little graded Brauer group. 

\begin{remark}
Let $(\chi,x) \in \Br_s(F)$. Then its inverse with respect to the group law is
$(\chi,x + \chi\cup\chi )$.
\end{remark}

\subsection{Wall homomorphism}
Equation \eqref{eq.PlusTimes} implies that 
$Q \mapsto [C(Q)]_s$ induces a
homomorphism from the even Grothendieck-Witt group to the little graded
Brauer group $\Br_s(F)$. 
As an immediate consequence of Proposition \ref{prop.ClEx}, 
one has 

\begin{proposition}\label{prop.WaEx}
\hfill 
\begin{enumerate}
\item For any quadratic \'etale algebra $E$ and any $a \in F^\times$ 
\[
[C(aQ_E)]_s=(\chi_E,[D(E,a)]). 
\] 
\item For any quaternion algebra $D$ 
\[
[C(Q_D)]_s=(0,[D]). 
\]
\end{enumerate}
\end{proposition}
Here, in accordance with our convention of using additive notation, $0$
stands for the trivial sign character on $\Gamma$. Taking $E=F\times F$ in
the first part of the proposition, one sees that
$[C(Q)]_s$ is trivial when $Q$ is a hyperbolic plane. Therefore 
$Q \mapsto [C(Q)]_s$ induces a
homomorphism 
\[
W_0(F) \to \Br_s(F). 
\] 
This homomorphism was studied by Wall; we will 
 denote it by  
\[
\Wall:W_0(F) \to \Br_s(F). 
\] 
Concretely, $\Wall(Q)$ is a pair $(\Wall_1(Q),\Wall_2(Q))$ with
$\Wall_1(Q) \in H^1(F,\mathbb F_2)$ and $\Wall_2(Q) \in \Br(F)$. 
The image of the Wall homomorphism $W_0(F) \to \Br_s(F)$ is equal to
the subgroup $\Br_2(F)_s$.

\subsection{$\Wall(aQ)$ versus $\Wall(Q)$} \label{sub.WaQWQ}

 We denote  by  $\{\cdot,\cdot\}$  the pairing  
\[
 H^1(F,\mathbb F_2) \otimes 
\bigl( F^\times/(F^\times)^2 \bigr)  \to \Br_2(F)
\]
 in Chapter 14 of \cite{Serre}. The class of $D(E, a)$ 
 in the Brauer group is then $\{\chi_E,a\}$. 
 The pairing $\{\cdot,\cdot\}$ is 
 reviewed in greater detail in subsection  \ref{sub.LocSymbol}. 

We have the following corollary of Proposition \ref{prop.WaEx} (1). 
In the corollary we view $\Br_2(F)$ as a subgroup of 
$\Br_2(F)_s$ via the injection $x \mapsto (0,x)$.  

\begin{lemma}\label{cor.WaSC}
For any even dimensional quadratic space $(V,Q)$  and any $a \in F^\times$
there is an equality 
\[
\Wall(aQ)-\Wall(Q)=\{\chi_Q,a \} \in \Br_2(F), 
\] 
where $\chi_Q$ is the sign character $\Wall_1(Q)$. Concretely, this means
that \begin{align*}
\Wall_1(aQ)&=\Wall_1(Q),  \\  
\Wall_2(aQ)&=\Wall_2(Q)+
\{\chi_Q,a \}.
\end{align*} 
\end{lemma} 

\begin{proof}
This follows from Proposition \ref{prop.WaEx} (1), because any even
dimensional quadratic space is an orthogonal direct sum of ones of the
form $bQ_E$ ($E$ quadratic \'etale, $b \in F^\times$). 
\end{proof}

\section{Review of the Clifford group $Cl(Q)$ of $Q$} 
We consider a nondegenerate 
quadratic space $(V,Q)$ over a field $F$.  We write $C=C(Q)$
for the Clifford algebra of
$(V,Q)$. It is a graded (i.e., $\mathbb F_2$-graded) algebra, and contains
$V$ as a  linear subspace of $C_1$. There is a canonical 
anti-involution $x \mapsto x_t$ on $C$ such that  
$(v_1v_2\dots v_m)_t := v_m\dots v_2v_1$ for  
$v_1,v_2,\dots, v_m \in V$. 

We write $ C^{\times_s}$ for the group of invertible homogeneous elements in $C$. 
For $x \in C^{\times_s}$ we define an automorphism $\Int_s(x)$ of the graded
algebra $C$ by
putting 
\[
\Int_s(x)(y) := (-1)^{\deg(x)\deg(y)} xyx^{-1}
\]
for all homogenous $y \in C$. There is a short exact sequence 
\[
1 \to F^\times \to C^{\times_s} \xrightarrow{\Int_s} \Aut_s(C) \to 1, 
\]
where $\Aut_s(C)$ denotes the group of automorphisms of the graded algebra $C$.

By functoriality any element $u$ in the orthogonal group $O(Q)$ induces an
automorphism $u_*$ of the graded algebra $C$, and $u_*$ preserves the
subspace $V$ of $C_1$. It is easy to see that $u \mapsto u_*$ yields
an isomorphism 
\[
O(Q) \to \{\theta \in \Aut_s(C) \, : \, \theta(V)=V  \}, 
\]
with inverse $\theta \mapsto \theta|_V$.

The Clifford group $Cl(Q)$ is defined to be 
$\{ x \in C^{\times_s} \, : \, \Int_s(x) (V)=V \}$. 
There is a canonical
homomorphism $\rho:Cl(Q) \to O(Q)$; it sends $x \in Cl(Q)$ to
$\Int_s(x)|_V$. The homomorphism $\rho$ is surjective, so there is an exact
sequence
\begin{equation}\label{SES.ClO}
1 \to \mathbb G_m \to Cl(Q) \xrightarrow{\rho} O(Q) \to 1
\end{equation}
of algebraic groups over $F$.  
The restriction of $\deg:C^{\times_s} \to \mathbb F_2$ to $Cl(Q)$ is
trivial on
$\mathbb G_m$, and hence induces a homomorphism of algebraic groups 
\begin{equation}
\deg:O(Q) \to \mathbb F_2,
\end{equation}
whose kernel is the special orthogonal group $SO(Q)$. Recall that $Cl(Q)$
contains $\{v \in V : Q(v) \ne 0 \}$, and that for $v$ in this subset
$\rho(v)$ is equal to reflection $r_v$ in $v$. 

There is a canonical homomorphism 
\begin{equation}\label{eq.DefOfN}
N:Cl(Q) \to \mathbb G_m
\end{equation} 
defined by $N(x)=x_tx$, where $x\mapsto x_t$ is the 
anti-involution defined near the beginning of this subsection. 
The restriction of $N$ to $\mathbb G_m 
\subset  Cl(Q)$ is the squaring map. 

\section{Wall invariants versus Hasse-Witt invariants}
We work over an arbitrary field $F$, with separable closure $\overline{F}$
and absolute Galois group $\Gamma$. In this appendix we again tacitly assume 
that all quadratic spaces we consider are nondegenerate. 

\subsection{Relative Hasse-Witt invariants} 
Let $n$ be a positive integer. Consider two $n$-dimensional
quadratic spaces $(V,Q)$ and $(V',Q')$. The difference between them is
measured by an element 
\[
\inv(Q',Q) \in H^1(F,O(Q)).
\] 
One obtains a specific 
$1$-cocycle representing $\inv(Q',Q)$ by choosing an
$\overline{F}$-isomorphism $u:(V,Q) \to (V',Q')$ and then putting 
$a_\sigma:= u^{-1}\sigma(u)$. 

We write 
\[ 
\partial:H^1(F,O(Q)) \to H^2(F,\mathbb G_m)
\]
for the connecting 
map for the short exact sequence \eqref{SES.ClO}. 
Recall from equation \eqref{eq.DefOfN} the 
homomorphism $N: Cl(Q) \to \mathbb G_m$, 
which restricts to the squaring map on $\mathbb G_m \subset
Cl(Q)$. Using it, one shows easily that the image of $\partial$ is contained in the subgroup $\Br_2(F)$ of the Brauer group.  

We write
\[ 
\deg_*:H^1(F,O(Q)) \to H^1(F,\mathbb F_2)
\]
 for the map induced by the
homomorphism $\deg:O(Q) \to \mathbb F_2$. 

\begin{definition}
The relative Hasse-Witt invariants of $(Q',Q)$ are defined as follows. 
\begin{enumerate}
\item $HW_1(Q',Q):=\deg_*\bigl(\inv(Q',Q)\bigr) \in H^1(F,\mathbb F_2)$, 
\item $HW_2(Q',Q):=\partial\bigl(\inv(Q',Q)\bigr) \in \Br_2(F)$, 
\item $HW(Q',Q):=\bigl(HW_1(Q',Q),HW_2(Q',Q)\bigr) \in\Br_2(F)_s$. 
\end{enumerate}
\end{definition}

The following facts are helpful when one needs to manipulate 
Hasse-Witt invariants. 

\begin{lemma}\label{lem.HWvs Wall}
For any  three quadratic spaces $(V,Q)$, $(V',Q')$, $(V'',Q'')$  
of the same  dimension  there are equalities 
 \begin{enumerate}
\item
$HW(Q'',Q)=HW(Q'',Q')+HW(Q',Q)$, 
\item
$HW(Q',Q)=\Wall(Q)-\Wall(Q')$,
\end{enumerate}
in the abelian group $\Br_2(F)_s$. 
\end{lemma}

\subsection{Comparison of $\Wall_1(Q)$, $\Wall_2(Q)$ with 
the discriminant and classical Hasse-Witt invariant }

In this paper we are primarily concerned with $\Wall(Q)$ and 
$HW(Q',Q)$, which have the advantage of working in a uniform 
way for all fields, even those of characteristic $2$. 
In Appendix \ref{app.FrohArb}, however, we 
will need to recast a result of Fr\"ohlich, who excluded characteristic 
$2$ and used invariants that make sense only in that context. 
This subsection reviews the dictionary between the two systems.

 In this
subsection we assume that  the characteristic of $F$ is not
$2$.   For $a_1,\dots,a_n \in F^\times$ we consider 
the quadratic form 
\[
Q(x_1,\dots,x_{n})
= a_1x_1^2 +a_2x_2^2 + \dots +a_{n}x_{n}^2.
\]
We then have two classical invariants of $Q$, the discriminant 
\[
\disc(Q) = \prod_{j=1}^n a_j \in F^\times/(F^\times)^2 = 
H^1(F,\mathbb F_2),
\]
and the Hasse-Witt invariant 
\[
hw(Q)=\sum_{1\le i < j \le n} a_i \cup a_j \in 
H^2(F,\mathbb F_2) = \Br_2(F).
\]

The invariants $\disc(Q)$ and $hw(Q)$ are 
 used in \cite{Froh}. 
In the case of quadratic forms of even rank, 
the Wall classes $\Wall_1(Q)$, $\Wall_2(Q)$ 
are different from, but carry the same information as, the invariants $\disc(Q)$
and $hw(Q)$. In order to formulate
a precise statement, we define an element $SW(Q) \in \Br_2(F)_s$ by 
putting $SW(Q)=(\disc(Q),hw(Q))$. The notation $SW$ is meant to  
indicate that this is a truncated version of 
Delzant's \cite{Delzant} total Stiefel-Whitney class of $Q$, which 
lies in the multiplicative group 
\[
1 + H^1(F,\mathbb F_2) + H^2(F,\mathbb F_2) + H^3(F,\mathbb F_2) 
+ \dots 
\]

We also define a specific element $z$
in $\Br_2(F)_s$ by $z:=(z_1,z_2)$ with $z_1=-1_F \in F^\times/(F^\times)^2
=H^1(F,\mathbb F_2)$ and $z_2=0 \in \Br_2(F)$. Observe 
that $z$ is the class $SW(Q_{F \times F})$ of the hyperbolic plane 
$(F \times F,Q_{F \times F})$. 

\begin{lemma}[Wall] \label{lem.WallComp}
Let $(V,Q)$ be a quadratic space of 
dimension $2n$. Then there is an equality 
\[
SW(Q)+\Wall(Q)=nz
\] 
in the abelian group $\Br_2(F)_s$. 
\end{lemma} 

\begin{corollary}[Wall]\label{lem.ClassConversion}
Let $(V,Q)$, $(V',Q')$ be two quadratic spaces of the same 
even dimension. 
Then there is an equality 
\[
HW(Q',Q)=SW(Q')-SW(Q)
\] 
in the abelian group $\Br_2(F)_s$.
\end{corollary}

\begin{proof}
Corollary \ref{lem.ClassConversion} follows from 
Lemma \ref{lem.WallComp} together with 
Lemma \ref{lem.HWvs Wall}(2). 
\end{proof}

\section{Review of Weil indices} In this appendix we assume our ground field $F$ is
local and
fix  a nontrivial additive character $\psi$ on $F$. We again make the 
tacit assumption that all the quadratic spaces we consider are 
nondegenerate. 

\subsection{Weil indices for even dimensional quadratic spaces} 
For any quadratic space $(V,Q)$  its Weil index
$\gamma(Q,\psi)$ is  defined. It is a complex eighth root of unity. When
$V$ is even dimensional the Weil index is actually a fourth root of unity. 

Weil \cite{Weil} proved the following properties of the Weil index. 
\begin{proposition}[Weil]\label{prop.Weil}
\hfill
\begin{enumerate}
\item $\gamma(Q_1 \oplus Q_2,\psi)=\gamma(Q_1,\psi)\gamma(Q_2,\psi)$. 
\item $\gamma(Q,\psi)=1$ when $Q$ is a hyperbolic plane. 
\item $Q \mapsto \gamma(Q,\psi)$ induces a unitary character $\gamma_\psi$ 
on the abelian group $W_0(F)$. 
\item For any quaternion algebra $D$ the Weil index of $Q_D$ is given by 
\begin{equation*}
\gamma(Q_D,\psi)=\begin{cases}
-1 &\text{ if $D$ is not split,}\\
1 &\text{ if $D$ is  split.}
\end{cases}
\end{equation*}
\end{enumerate}
\end{proposition}

Jacquet-Langlands \cite[Lemma 1.2]{JL} proved the 
following result. As in Tate's article \cite{Tate} we
write $\epsilon_L$ for Langlands's version of local epsilon
factors.  Thus $\epsilon_L(\rho,\psi)$ is defined for any continuous
complex representation $\rho$ of $\Gamma$ 
(or, more generally, of the Weil group
of $F$). 
\begin{proposition}[Jacquet-Langlands]\label{prop.JL}
Let $E$ be a quadratic \'etale algebra, and let $\chi_E$ be the associated
sign character on $\Gamma$. Then 
\[
\gamma(Q_E,\psi)=\epsilon_L(\chi_E,\psi).
\]
\end{proposition}

Proposition \ref{prop.Weil} (4) and Proposition \ref{prop.JL} 
can be combined into a single statement. For this we define objects
$\chi_Q$ and $\zeta_Q$ associated to any even dimensional quadratic space
$(V,Q)$. First, $\chi_Q$ is by definition the sign character on $\Gamma$
obtained as $\Wall_1(Q)$. Second, $\zeta_Q$ is given by 
 \begin{equation*}
\zeta_Q=\begin{cases}
-1 &\text{ if $\Wall_2(Q)$ is nontrivial,}\\
1 &\text{ if $\Wall_2(Q)$ is  trivial.}
\end{cases}
\end{equation*} 
We should recall that $\Wall_2(Q)$ lies in $\Br_2(F)$, 
which for a local field is
\begin{enumerate}
\item cyclic of order $2$ except when $F$ is isomorphic to $\mathbb C$, 
\item trivial when $F$ is isomorphic to $\mathbb C$. 
\end{enumerate}

\begin{proposition}\label{prop.combo}
Let $(V,Q)$ be an even dimensional quadratic space. Then 
\[
\gamma(Q,\psi)=\epsilon_L(\chi_Q,\psi) \,\zeta_Q.
\]
\end{proposition}
\begin{proof}
Unless $F$ is isomorphic to $\mathbb R$ the Wall homomorphism $W_0(F) \to
\Br_2(F)_s$ is an isomorphism, and the proposition we are now proving
follows directly from Proposition \ref{prop.Weil} (4) and Proposition
\ref{prop.JL}. The proof in the real case is the same, once one observes
that the unitary character
$\gamma_\psi$ on 
$W_0(\mathbb R)$ is trivial on the kernel of the Wall homomorphism. Indeed
the kernel of the Wall homomorphism is the infinite cyclic group generated
by $Q_D \oplus Q_D$, where $D$ is the Hamiltonian 
quaternion algebra, and 
$\gamma_\psi(Q_D \oplus Q_D)=(-1)^2=1$.  
\end{proof}

\section{Orthogonal representations of $\Gamma$} 

Let $\rho$ be a (continuous) finite dimensional complex representation of
$\Gamma$. Assume that $\rho$ is orthogonal. Then the Stiefel-Whitney
classes $w_n(\rho) \in H^n(F,\mathbb F_2)$ are defined (see Deligne \cite{Del}). 
In fact they are even defined for virtual orthogonal 
complex representations, 
i.e.~formal differences $ \rho = \rho_1 - \rho_2$ of two 
orthogonal complex representations $\rho_1$ and $\rho_2$.

We write $\bar w_2(\rho)$ for the element in $\Br_2(F)$ 
obtained as the
image of $w_2(\rho)$ under the homomorphism 
$H^2(F,\mathbb F_2) \to
H^2(F,\overline F^\times)$ induced by the homomorphism 
$\mathbb F_2 \to
\overline F^\times$ that sends $n$ to $(-1)^n$. Of course 
$\bar w_2(\rho)$
is always trivial when the characteristic of $F$ is $2$. 
Finally, we put $\overline{SW}(\rho)=(w_1(\rho),
\bar w_2(\rho)) \in \Br_2(F)_s$. 

Following Deligne \cite{Del}, we want to write a formula for the local epsilon
factor
$\epsilon_L(\rho,\psi)$ in terms of the first and second Stiefel-Whitney
classes of $\rho$.  To do so we define objects
$\chi_\rho$ and $\zeta_\rho$. First, $\chi_\rho$ is by definition the sign
character on $\Gamma$ obtained as $w_1(\rho)$. More concretely,
$\chi_\rho$ is the sign character $\det(\rho)$. 
Second, $\zeta_\rho$ is given
by 
 \begin{equation*}
\zeta_\rho=\begin{cases}
-1 &\text{ if  $\bar w_2(\rho)$ is nontrivial,}\\
1 &\text{ if  $\bar w_2(\rho)$ is trivial.}
\end{cases}
\end{equation*} 

We then have the following reformulation of Deligne's result: 
\begin{proposition}[Deligne] \label{prop.Deligne}
Let $\rho$ be a virtual orthogonal representation of
$\Gamma$. Then 
\[
\epsilon_L(\rho,\psi)=\epsilon_L(\chi_\rho,\psi) \,\zeta_\rho.
\]
\end{proposition}
\begin{proof}
Deligne proves this 
statement when $\rho$ has dimension $0$ and trivial determinant. In that case,
since
$\epsilon_L(\chi,\psi)=1$ when $\chi$ is the trivial character on
$\Gamma$, the  proposition simply says that 
\[
\epsilon_L(\rho,\psi)=\zeta_\rho.
\]
The general case
follows from Deligne's result for the virtual representation
$\rho' = \rho-\det(\rho)-V$, where $V$ is a direct sum of $\dim(\rho)-1$ copies of
the trivial representation of $\Gamma$. The point is that $w_2(\rho') = w_2(\rho)$.\end{proof}

\section{Relation between Weil indices and local epsilon factors}

Comparing Propositions \ref{prop.combo} and \ref{prop.Deligne},  
we arrive at the following conclusion. We again make the 
tacit assumption that the
quadratic spaces we consider are nondegenerate.

\begin{proposition}\label{thm.CompEpWeil}
 Let $(V,Q)$ be an even dimensional quadratic space, and let
$\rho$ be a virtual orthogonal complex representation of 
$\Gamma$. Assume further that
$\Wall(Q)=\overline{SW}(\rho)$. Then 
\[
\gamma(Q,\psi)=\epsilon_L(\rho,\psi). 
\]
\end{proposition}

\begin{corollary} \label{cor.Wep}
Let $(V,Q)$, $(\tilde V,\tilde Q)$ be two quadratic spaces of the same even
dimension, and let $\rho$ be a virtual orthogonal  
complex representation of $\Gamma$. 
If $HW(Q,\tilde Q)=\overline{SW}(\rho )$, then $\gamma(\tilde
Q,\psi)\gamma(Q,\psi)^{-1} =
\epsilon_L(\rho,\psi) $.
\end{corollary}

\begin{proof}
Apply the proposition to the quadratic form $\tilde Q \oplus (-Q)$. 
\end{proof}

\section{Fr\"ohlich theory in arbitrary characteristic}
\label{app.FrohArb}

Let $F$ be a field, $\overline{F}$ a separable closure, and
$\Gamma=\Gal(\overline{F}/F)$. We are again interested in the 
abelian groups 
$\Br_2(F) \subset \Br_2(F)_s$. 
We continue to make the tacit assumption that the 
quadratic spaces we consider are nondegenerate.

\subsection{The local symbol $\{\chi,b\}$} \label{sub.LocSymbol}

We need to review the local symbol $\{\chi,b\}$
(a standard reference is  \cite[Chapter 14]{Serre}) in greater detail. 
 For $\chi \in H^1(\Gamma,\mathbb F_2)$ and 
$b \in F^\times$, we
denote by $\{\chi,b\}$ the element in $\Br_2(F)$ obtained as the 
cup-product $b \cup \delta\chi$, where $\delta$ denotes the
connecting homomorphism $H^1(F,\mathbb F_2)
\to H^2(F,\mathbb Z)$ obtained from the 
short exact sequence  
\[
0 \to \mathbb Z \xrightarrow{2} \mathbb Z \to \mathbb F_2 \to 0
\] 
of (trivial) $\Gamma$-modules.  It is evident that 
 $\{\chi,b\}$ is additive in both $\chi$ and $b$. It follows that
$\{\chi,b^2\}=(2\chi,b)=(0,b)=0$, so that $(\chi,b)$ 
can be viewed as a pairing 
\begin{equation}\label{eq.Hilb}
H^1(\Gamma,\mathbb F_2) \otimes_{\mathbb Z}
F^\times/(F^\times)^2 \to \Br_2(F). 
\end{equation}

When the characteristic of $F$ is not $2$, Kummer theory identifies
 $F^\times/(F^\times)^2$ with $H^1(\Gamma,\mathbb F_2)$, and  
 we may view  the pairing above as the Hilbert pairing 
\[
F^\times/(F^\times)^2 \otimes_{\mathbb Z} F^\times/(F^\times)^2 \to
\Br_2(F)
\]
or  as the cup-product pairing  
\[
H^1(\Gamma,\mathbb F_2) \otimes_{\mathbb Z} H^1(\Gamma,\mathbb F_2) \to
H^2(\Gamma,\mathbb F_2)=\Br_2(F).
\]
All of this is explained in Serre's book. 

We are especially interested in 
 characteristic $2$, so the relevant pairing is
\eqref{eq.Hilb}. In order to formulate Fr\"ohlich's result 
in a way that  is valid even in characteristic $2$, 
we need to view
the pairing in an equivalent form. 
Let $U$ be an $\mathbb F_2$-vector space, and view $U$ as a $\Gamma$-module
with trivial $\Gamma$-action. There are  obvious
(cup-product) homomorphisms 
\[
H^n(\Gamma,\mathbb F_2) \otimes_{\mathbb Z} U \to H^n(\Gamma,U).
\]
We claim that these homomorphisms are isomorphisms. Indeed, this follows
from the fact that taking cohomology for a profinite group preserves
direct sums (because it preserves finite direct sums and filtered
colimits). Taking $U$ to be $F^\times/(F^\times)^2$, we find that 
\[
H^1(\Gamma,F^\times/(F^\times)^2)=H^1(\Gamma,\mathbb F_2) \otimes_{\mathbb Z} 
F^\times/(F^\times)^2, 
\] 
so that we may view the pairing \eqref{eq.Hilb} as a homomorphism 
\begin{equation}
\xi:H^1(\Gamma,F^\times/(F^\times)^2) \to \Br_2(F).
\end{equation}
When we do so, we have the equality 
\begin{equation}\label{eq.XiVsLocSymb}
\{\chi,b\} = \xi(\chi \otimes b). 
\end{equation}

We are going to use $\xi$ to extend Fr\"ohlich's result to fields of 
characteristic $2$. For such a field   
the squaring map is injective, so square roots are unique when they exist. 
The next result gives a more explicit description of $\xi$ in 
characteristic  $2$. 

\begin{lemma}\label{lem.xiFroh} 
Assume that the  field $F$ has characteristic $2$. 
 Let $\eta$ be a 
homomorphism $\Gamma \to F^\times/(F^\times)^2$ 
with open kernel.      
Choose a $1$-cochain
$u_\sigma$ of $\Gamma$ in $F^\times$ such 
that the image of $u_\sigma$
under $F^\times \twoheadrightarrow F^\times/(F^\times)^2$ 
is $\eta(\sigma)$.
The homomorphism property of $\eta$  implies 
that $u_\sigma u_\tau u_{\sigma\tau}^{-1}$ is 
a square in $F^\times$, and  
 $\xi(\eta)$ is represented by the $2$-cocycle 
\[
z_{\sigma,\tau}=\sqrt{u_\sigma u_\tau u_{\sigma\tau}^{-1}} 
\in F^\times \subset \overline{F}^\times.
\] 
\end{lemma}
 
\begin{proof}
Without loss of generality we may assume that $\eta$ is of the form 
$\chi \otimes a$ with $\chi \in H^1(\Gamma,\mathbb F_2)$ and 
$a \in F^\times$.
Concretely this means that $\eta(\tau)=a^{\chi(\tau)} \in
F^\times/(F^\times)^2$.

We need to compute $\xi(\eta)$. To obtain a cocycle version of the
connecting map 
$H^1(\Gamma, \mathbb F_2) \to H^2(\Gamma,\mathbb Z)$, one needs to choose
a set-theoretic section of $\mathbb Z \twoheadrightarrow \mathbb F_2$; our
choice will be the  liftings
$0,1
\in
\mathbb Z$ of
$0,1
\in
\mathbb F_2$. With this choice, $\xi(\eta) \in \Br_2(F)$ is represented by
the
$2$-cocycle $z'$ of $\Gamma$ in $F^\times \subset \overline{F}^\times$
given by 
\[
z'_{\sigma,\tau}=
\begin{cases}
a &\text{ if both 
$\chi(\sigma)$ and $\chi(\tau)$ are nontrivial,} \\ 
1 &\text{ otherwise.}
\end{cases}
\]

 Since the cohomology
class of $z_{\sigma,\tau}$ is obviously independent of the choice of
lifting $u_\sigma$, we are free to take 
\[
u_\sigma=a^{\dot\chi(\sigma)} \in F^\times,
\]  
where $\dot\chi(\sigma)$ is  the unique element of $\{0,1\} \subset
\mathbb Z$ lifting $\chi(\sigma) \in \mathbb F_2$. The $2$-cocycle 
\[
z_{\sigma,\tau}=\sqrt{u_\sigma u_\tau u_{\sigma\tau}^{-1}}
\]
then   coincides with $z'$, and the lemma is proved. 
\end{proof}

\subsection{Stiefel-Whitney classes for orthogonal representations in
arbitrary characteristic} 
Let $k$ be an algebraically closed field, and 
let $G$ be a profinite group. When the characteristic of $k$ is not $2$, Fr\"ohlich
attaches first and second
Stiefel-Whitney classes to orthogonal representations of
$G$ on quadratic spaces over $k$. In other words, given a quadratic space 
$(V,Q)$ over $k$, and a homomorphism $\varphi:G \to O(Q)$ with open kernel,
Fr\"ohlich defines
 an element $SW(\varphi)$ in  
the abelian group $1 +H^1(G,\mathbb F_2) + H^2(G,\mathbb F_2)$.  
The $i$-component of $SW(\varphi)$ ($i=1,2$) will be denoted by
$SW_i(\varphi)$. 
For $k=\mathbb C$ these classes are the traditional Stiefel-Whitney classes
$w_1,w_2$ that Deligne uses in
his paper on local epsilon factors for orthogonal representations.

Fr\"ohlich defines
the first Stiefel-Whitney class $SW_1(\varphi)$ to be the sign character
$\det\circ\varphi$.
At the moment we are assuming that the 
characteristic of $k$ is not $2$, so there is a short exact sequence 
\[
1 \to \mu_2(k) \to \widetilde{\Pin}(Q) \to O(Q) \to 1
\]
of  groups  with trivial $G$-action, giving rise to a connecting map 
\[
\partial: H^1(G,O(Q)) \to H^2(G,\mu_2(k))=H^2(G,\mathbb F_2), 
\]
and Fr\"ohlich defines $SW_2(\varphi)$ to be the image
of $\varphi$ under this connecting map. (Here  $\varphi$ is viewed as 
a  $1$-cocycle of $G$ in $O(Q)$.) 

Now we want to extend Fr\"ohlich's definitions to fields $k$ of characteristic $2$.
We put $SW_1(\varphi) := \deg \circ\varphi$, where $\deg$ is the degree
homomorphism from $O(Q)$
to $\mathbb F_2$. When the characteristic is not $2$, $\deg$ is equal to $\det$, so
our definition is compatible
with Fr\"ohlich's. We put  $SW_2(\varphi)=0$ when $k$ has characteristic $2$.  
There is no real content in this definition, but it will allow for a uniform
statement when we extend
Fr\"ohlich's result to characteristic $2$.

\subsection{Review of the spinor norm homomorphism}
\label{sub.RevSpNorm}  

 We consider a quadratic space $(V,Q)$ over $F$.  
We are interested in its orthogonal group $O=O(Q)$ and Clifford 
group $Cl=Cl(Q)$. 
There is a short exact sequence (see equation \eqref{SES.ClO})  
\[
1 \to F^\times \to Cl(F) \xrightarrow{\rho} O(F) \to 1,  
\] 
as well as a canonical character $N:Cl(F) \to F^\times$  
(see equation \eqref{eq.DefOfN}).

Recall that the spinor norm homomorphism 
\[ 
\delta:O(F) \to
F^\times/(F^\times)^2 
\]
is defined as follows. Let $x \in O(F)$ and
choose $\dot x \in Cl(F)$ such that $\rho(\dot x)=x$. Then take 
$\delta(x)$ to be the square-class of $N(\dot x)$. 

\subsection{Fr\"ohlich's result}\label{sub.FrohReview}

We want to extend a result of Fr\"ohlich \cite{Froh}  to fields 
$F$ of arbitrary characteristic.
Now consider an orthogonal representation $\varphi:\Gamma \to O(F)$. 
Once again $\Gamma$ is the absolute Galois group of $F$. 
There are two things we can do with $\varphi$. The first is to extend
scalars from $F$ to an algebraically closed field $k$ containing $F$. In
this way we obtain an orthogonal representation 
\[ 
\varphi_k:\Gamma \xrightarrow{\varphi} O(F) \hookrightarrow O(k)
\] 
 of $\Gamma$ over $k$. It is evident that the Stiefel-Whitney class
$SW(\varphi_k)$ is independent of the choice of $k$, so we may as well
abbreviate it to $SW(\varphi)$. We denote by $\overline{SW}(\varphi)$ the
image of $SW(\varphi)$ under the natural homomorphism from 
$1 +H^1(\Gamma,\mathbb F_2) + H^2(\Gamma,\mathbb F_2)$ to $\Br_2(F)_s$. 

The second thing we can do with $\varphi$ is to twist $(V,Q)$ (see subsection
\ref{sub.TwistQ} for a
review), obtaining
a new quadratic space $(V_{\varphi},Q_{\varphi})$ over $F$.  The element
$\inv(Q_{\varphi},Q) \in H^1(F,O)$ is obtained as the image of $\varphi$
under the natural map 
\[
H^1(\Gamma,O(F)) \to H^1(F,O). 
\] 
We can then consider the relative Hasse-Witt invariant $HW(Q_{\varphi},Q)
\in 
\Br_2(F)_s$. 

The following result of Fr\"ohlich  makes use of the spinor norm map 
$\delta:O(F) \to F^\times/(F^\times)^2$ and the homomorphism 
$\xi:H^1(\Gamma,F^\times/(F^\times)^2) \to \Br_2(F)$ defined earlier. 

\begin{proposition}[Fr\"ohlich]\label{prop.FrohExt}
$\overline{SW}(\varphi)$ and $HW(Q_{\varphi},Q)$ differ by the element
$\xi(\delta \circ \varphi)$ in the subgroup $\Br_2(F)$ of $\Br_2(F)_s$. 
\end{proposition}

\begin{proof} 
When the characteristic is not $2$, an equivalent result  is  proved by 
Fr\"ohlich. 
When his result is restated in terms of relative Hasse-Witt invariants, 
it becomes the slightly different looking statement given here, 
as one can check using Corollary \ref{lem.ClassConversion}.  

Now we deal with the case when the characteristic is $2$. 
 It is clear that
$SW_1(\varphi)=HW_1(Q_{\varphi},Q)$. So the point is really to show that 
the elements 
 $\overline{SW}_2(\varphi), HW_2(Q_{\varphi},Q) \in \Br_2(F)$ differ by
$\xi(\eta)$, with
$\eta:=\delta \circ \varphi$.

Now 
$SW_2(\varphi)$ and $\overline{SW}_2(\varphi)$ are $0$,  
so we need to prove that
\[
HW_2(Q_{\varphi},Q)=\xi(\eta).
\]
To do so we make use of  the  exact sequence 
\[
1 \to F^\times \to Cl(F) \xrightarrow{\rho} O(F) \to 1.
\]
 We choose a finite Galois extension $K/F$ such
that $\varphi$ is inflated from a homomorphism 
$\Gal(K/F) \to O(F)$. For $\sigma \in
\Gal(K/F)$ we choose $\dot\varphi(\sigma) \in Cl(F)$ such that
$\rho(\dot\varphi(\sigma))=\varphi(\sigma)$. Since these liftings are
fixed by the Galois group, the Hasse-Witt invariant $HW_2(Q_{\varphi},Q)$
is represented by the $2$-cocycle 
\begin{equation}\label{eq.zANDphi}
z_{\sigma,\tau}=\dot\varphi(\sigma) \, \dot\varphi(\tau) \,
(\dot\varphi(\sigma\tau))^{-1}.
\end{equation}

Next we need to calculate $\xi(\eta)$. 
Applying the character $N$ to the equation \eqref{eq.zANDphi}, 
we find that 
\[
z_{\sigma,\tau}^2=u_\sigma \, u_\tau \,
u_{\sigma\tau}^{-1},
\]
with $u_\sigma \in F^\times$ defined by $u_\sigma=
N(\dot\varphi(\sigma))$.  
Now $\eta(\sigma)$ is equal to the square class of 
$u_\sigma$. So Lemma \ref{lem.xiFroh} implies that $\xi(\eta)$
is represented by the $2$-cocycle 
\[
\sqrt{u_\sigma u_\tau u_{\sigma\tau}^{-1}}=z_{\sigma,\tau}. 
\] 
Thus $z_{\sigma,\tau}$ represents both $HW_2(Q_{\varphi},Q)$ 
and $\xi(\eta)$, and the
proof is complete.  
\end{proof}

\subsection{Spinor norms of reflections} \label{sub.SpinNormRefl}

When applying Fr\"ohlich's result, it is useful to understand 
spinor norms of reflections. 
For $v \in V$ with $Q(v) \ne 0$, the element $v$ in the Clifford algebra 
$C(Q)$ is an odd element in the Clifford group.
It projects to the reflection $r_v \in O(Q)$. 
So the spinor norm of $r_v$ is $N(v) = v^2 = Q(v)$.

\subsection{Review of twisting of quadratic forms}\label{sub.TwistQ}

Let $V$ be an $F$-vector space. Then $\overline{F} \otimes_F V$ 
carries a semilinear smooth $\Gamma$-action, given
by $\sigma(\alpha \otimes v)=\sigma(\alpha) \otimes v$, 
having $V \subset
\overline{F} \otimes_F V$ as fixed point set.
Conversely, given a vector space $U$ over $\overline{F}$, equipped with a
semilinear smooth
action of $\Gamma$, we obtain an $F$-vector space 
$V = U^\Gamma$ and a canonical
isomorphism $\overline{F} \otimes_F V \to U$.
Informally, then, giving an $F$-vector space is the same as giving an
$\overline{F}$-vector space equipped with a semilinear
smooth Galois action. This is Galois descent theory.  

In particular, given an $F$-vector space $V$ and a continuous $1$-cocycle $\sigma
\mapsto \theta_\sigma$ of $\Gamma$
in $\Aut(V)(\overline{F})$, we get a new semilinear action $\sigma_*:=\theta_\sigma
\sigma$, and so get a new $F$-vector space
\[
V_* = \{ v \in \overline{F} \otimes_F V : \theta_\sigma(\sigma(v)) = v \quad
\forall\, \sigma \in \Gamma \}
\]
having the same dimension as $V$. One says that $V_*$ is a \emph{twist} of $V$. 

Now suppose that we have a quadratic form $Q$ on $V$, and suppose further that our
$1$-cocycle $\theta_\sigma$
takes values in the subgroup $O(Q)(\overline{F})$ of $\Aut(V)(\overline{F})$.
Because $Q$ is defined over $F$, it is fixed by $\sigma$.
Because $\theta_\sigma$ fixes $Q$, the action of $\sigma_*$ also fixes $Q$. In
particular, the function
$Q:\overline{F} \otimes_F V \to \overline{F}$ is $F$-valued on $V_*$ as well as
$V$.  In summary, twisting a quadratic space $(V,Q)$ 
actually means that we twist the 
$F$-structure on $V$, while leaving the quadratic form unchanged. 
Nevertheless, we speak of ``twisting the quadratic form $Q$ by 
the $1$-cocycle $\theta$'' and denote the twisted form by $Q_\theta$.

\section{Relative Hasse-Witt invariants for twists by $1$-cocycles in tori}

Let $(V,Q)$ be a nondegenerate quadratic space over $F$ 
with $V$ even dimensional. Let $T$ be a torus, and suppose that 
we are given an action of $T$ on $V$ by orthogonal 
transformations. In other words
we are given
an $F$-homomorphism from $T$ to the orthogonal group of $V$.  
Finally, suppose that we are given a
$1$-cocycle $t$ of $\Gamma$ in $T$. We can then twist 
$(V,Q)$ by $t$, obtaining a 
new quadratic form $Q_t$. 

Now we are going to use $t$ to define an element in $\Br_2(F)$. 
We begin by defining an extension
$U$ of $T$ by $\mu_2$. Such extensions are classified by 
$\Gamma$-invariant
elements in the $\mathbb F_2$-vector space
$X^*(T)/2X^*(T)$. We now use the orthogonal action of 
$T$ on $V$ to define a
canonical such element, which we will denote by $\lambda_V$.

Let us write the dimension of $V$ as $n=2m$. 
Because $V$ is self-contragredient as representation of $T$,   
we have $\dim V_\lambda = \dim V_{-\lambda}$ for each 
$\lambda \in X^*(T)$. (Here $V_\lambda$ is the weight space
of $\lambda$ in $V$; it is defined over the separable closure.) 
So there exist   $\lambda_1,\dots,\lambda_m \in X^*(T)$ such 
that the $n$ characters of $T$  occurring in $V$ are 
$
\lambda_1,\dots,\lambda_m,-\lambda_1,\dots,-\lambda_m.
$
The $m$-tuple $(\lambda_1,\dots,\lambda_m)$ is 
well-defined up to permutations and  sign changes. Therefore
the sum $\lambda_V:=\lambda_1 + \dots + \lambda_m$ 
yields a well-defined element of
$X^*(T)/2X^*(T)$. Our assumption
that the torus action is defined over $F$ implies 
that $\lambda_V$ is a
$\Gamma$-invariant element in $X^*(T)/2X^*(T)$.

Now, as we have already said, to $\lambda_V$ is 
associated an extension 
\[
1 \to \mu_2 \to U \to T \to 1. 
\]
This sequence of $F$-groups is exact in the fppf topology, 
so from it we obtain a connecting homomorphism 
\[
\partial_U : H^1_{\fppf}(F,T) \to H^2_{\fppf}(F,\mu_2).
\]
The source coincides with the Galois cohomology group 
$H^1(F,T)$, and the target coincides with $\Br_2(F)$, so
we may also view $\partial_U$ as a homomorphism 
\begin{equation}\label{eq.partial_U}
 \partial_U : H^1(F,T) \to \Br_2(F). 
\end{equation}

\begin{lemma} \label{lem.HWtorus} 
$HW(Q_t,Q)$ is equal to the image of $t$ under the map \eqref{eq.partial_U}. Here,
as usual, we are viewing
$\Br_2(F)$ as a subgroup of $\Br_2(F)_s$. 
\end{lemma}

\begin{proof}

We must prove that 
\begin{enumerate}
\item 
$HW_1(Q_t,Q) = 0$,
\item 
$HW_2(Q_t,Q) = \partial_U(t)$. 
\end{enumerate}

Without loss of generality we may assume that $T$ is a maximal torus in $O(Q)$. 
The first item follows from the fact that $T$  lies in the kernel $SO(Q)$ 
of the degree homomorphism $O(Q) \to \mathbb F_2$. The second item follows from 
the well-known fact that the element of $X^*(T)/2X^*(T)$ associated to 
\[
1 \to \mu_2 \to \Spin(Q) \to O(Q) \to 1
\]
is $\lambda_V$. 
\end{proof}

\begin{corollary}\label{cor.EpTor}
Now assume that $F$ is local, so that Weil indices are defined. Then 
\[
\gamma(Q_t,\psi) \gamma(Q,\psi)^{-1} = 
\partial_U(t) \in \Br_2(F) = \{ \pm 1
\}.
\]
\end{corollary}

\begin{proof}
This follows from Lemma \ref{lem.HWtorus} together with 
Lemmas  \ref{prop.combo} and \ref{lem.HWvs Wall}(2). 
\end{proof}

\section{Standard quadratic form $Q_2$ on 
$\Lie(\mathbb G_{\ssc})$ over $\mathbb Z$} \label{app.StQuadForm}

\subsection{Notation}

In this appendix we consider a split group 
$\mathbb G_{\ssc}$ over $\mathbb Z$ that is semisimple
and simply connected.  In it we choose a 
split maximal torus $\mathbb T_{\ssc}$, and  a 
Borel subgroup $\mathbb B_{\ssc}$ 
containing $\mathbb T_{\ssc}$. For each simple root 
$\alpha$ of $\mathbb T_{\ssc}$ we choose an isomorphism
$\eta_\alpha$ from the additive group $\mathbb G_a$ 
to the root subgroup $\mathbb G_\alpha$. Thus 
 $(\mathbb G_{\ssc},\mathbb B_{\ssc},
 \mathbb T_{\ssc},\{\eta_\alpha\})$ is 
a pinned group over $\mathbb Z$. Moreover, 
we assume that the root system  $R$ 
of $\mathbb T_{\ssc}$ in $\mathbb G_{\ssc}$ is irreducible. 
We write $Z(\mathbb G_{\ssc})$ for the center
of $\mathbb G_{\ssc}$, and $\mathbb G_{\ad}$ for the 
adjoint group $\mathbb G_{\ssc}/Z(\mathbb G_{\ssc}) $. 

We write $W$ for the Weyl group of $\mathbb T_{\ssc}$ 
in $\mathbb G_{\ssc}$, and put
\[
\Omega_0 := \Aut(\mathbb G_{\ssc},\mathbb B_{\ssc},\mathbb
T_{\ssc},\{\eta_\alpha\}).
\]
It is well-known that $\Aut(\mathbb G_{\ssc}) = 
\mathbb G_{\ad} \rtimes \Omega_0$. 
The group $\Omega := W \rtimes \Omega_0$ 
acts on $\mathbb T_{\ssc}$.

\subsection{Standard quadratic form $Q_1$ on 
$\Lie (\mathbb T_{\ssc})$}

We put 
\[
\mathfrak a := X_*(\mathbb T_{\ssc})_{\mathbb Q} = 
X_*(\mathbb T_{\ad})_{\mathbb
Q}.
\]
 So we are now viewing 
\[
X_*(\mathbb T_{\ssc})  \subset X_*(\mathbb T_{\ad}) 
\]
 as an inclusion of lattices in a common rational vector space 
 $\mathfrak a$. 

Because the root system $R$ is irreducible, the 
space of $W$-invariant quadratic forms
on $\mathfrak a$ is one dimensional. As basis 
element we choose the unique 
$W$-invariant quadratic form $Q_1$  whose 
values on coroots lie in the set $\{1,\ell \}$, where 
 \[
\ell = 
\begin{cases}
1 &\text{ if $R$ is of type $A,D,E$}, \\
2 &\text{ if $R$ is of type $B_n,C_n,F_4$ with $n \ge 2$}, \\
3 &\text{ if $R$ is of type $G_2$}. \\
\end{cases}
\]
So, when all  coroots have the same 
length, $Q_1$ takes the value $1$ on  them.  
 Otherwise $Q_1$ takes the value $1$ on short coroots, and 
takes the value $\ell$ on long coroots. 
The action of $\Omega$ preserves $Q_1$.  
We put $\ell(\alpha^\vee):=Q_1(\alpha^\vee) \in \{1,\ell \}$. 
 
Observe that $\ell$ is the same  
for $R^\vee$ as it is for $R$. The bijection $\alpha \mapsto \alpha^\vee$ from $R$
to $R^\vee$ sends long roots to
short coroots, and vice versa, so there is an equality 
\begin{equation}\label{eq.rtcrt}
\ell(\alpha^\vee)\ell(\alpha) = \ell. 
\end{equation}

Now polarize $Q_1$ to obtain an  
$\Omega$-invariant symmetric bilinear form $B_1$ on 
$\mathfrak a$. Thus $B_1$ and $Q_1$ are related by the identities 
 $B_1(x,y)=Q_1(x+y)-Q_1(x)-Q_1(y)$ and  
 $Q_1(x)=\frac{1}{2}B_1(x,x)$.   We may also view $B_1$  
 as an isomorphism
$B_1:\mathfrak a \to \mathfrak a^*$. It will always 
be clear from context which 
interpretation of $B_1$ is meant; the two interpretations 
are related by the identity
\[
B_1(x,y) =\langle B_1(x), y \rangle. 
\]
(We always use $\langle \cdot,\cdot \rangle$ to denote the canonical pairing between a module and its linear dual.)

From now on we use $B_1$ to
identify $\mathfrak a^*$ with $\mathfrak a$.  This allows us to view roots
as elements in
$\mathfrak a$.   

\begin{lemma}\label{lem.ellalph} 
The following statements hold. 
\begin{enumerate}
\item There is an equality $\alpha^\vee=\ell(\alpha^\vee) \alpha \in \mathfrak a$.\item There is an equality  $\ell
Q_1(\alpha)=\ell(\alpha)$. 
\end{enumerate}
\end{lemma}

\begin{lemma} \label{lem.intQ} 
The  integrality properties of $Q_1$ and $B_1$ are as follows.  
\begin{enumerate}
\item $Q_1$ takes integral values on the lattice $X_*(T_{\ssc})$ in
$\mathfrak a$. 
\item $B_1(x,y) \in \mathbb Z$ when $x \in X_*(T_{\ssc})$, $y \in
X_*(T_{\ad})$. 
\item $\ell Q_1$ takes integral values on the lattice $X^*(T_{\ad})$ in
$\mathfrak a$. 
\item $\ell B_1(x,y) \in \mathbb Z$ when $x \in X^*(T_{\ssc})$, $y \in
X^*(T_{\ad})$. 
\end{enumerate}
\end{lemma} In the third and fourth parts, one must bear in mind  
that we have identified 
$\mathfrak a^*$ with $\mathfrak a$, so that  
$X^*(T_{\ssc})$, $X^*(T_{\ad})$ may be viewed as lattices in
$\mathfrak a$.

\subsection{The standard quadratic form $Q_2$ on $\Lie(\mathbb G_{\ssc})$}
\label{sub.Q_2}

The space of $\mathbb G_{\ad}$-invariant 
quadratic forms on $\Lie(\mathbb G_{\ssc})$ 
is an abelian group that is free of rank $1$.
It has a unique generator $Q_2$  which restricts to 
$Q_1$ on $\Lie(\mathbb T_{\ssc}) = 
X_*(\mathbb T_{\ssc})$. Moreover $Q_2$ is invariant under
all automorphisms of $\mathbb G_{\ssc}$. 

This is well-known, but we need to recall how it works. 
Start by extending scalars
to $\mathbb Q$.
The space of invariant quadratic forms 
on $\Lie(\mathbb G_{\ssc})_\mathbb Q$ is
$1$-dimensional,
and any nonzero invariant form (e.g.~the Killing form) 
is nondegenerate over $\mathbb Q$.
Also, any invariant form is actually invariant under 
all automorphisms of $\mathbb G_{\ssc}$,
since this is obviously true for the Killing form. 

Now we decompose the Lie algebra of $\mathbb G_{\ssc}$ as 
\begin{equation}\label{eq.DSDgtv}
\Lie(\mathbb G_{\ssc}) = \Lie(\mathbb T_{\ssc}) \oplus \mathbb V,
\end{equation}
where 
\begin{equation} \label{eq.DSDvvv}
\mathbb V = \bigoplus_{\alpha > 0} \bigl(\Lie(\mathbb G_\alpha) 
\oplus \Lie(\mathbb G_{-\alpha})\bigr).
\end{equation}
Both  \eqref{eq.DSDgtv} and  \eqref{eq.DSDvvv} are orthogonal 
decompositions with respect to any invariant quadratic form. 

The restriction 
to $\Lie(\mathbb T_{\ssc}) = X_*(\mathbb T_{\ssc})$ 
of any nonzero invariant form
is $W$-invariant, hence
is a nonzero rational multiple of $Q_1$. So it is clear that 
there exists a unique invariant
quadratic form $Q_2$ on 
$\Lie(\mathbb G_{\ssc})_\mathbb Q$  whose restriction to 
$\Lie(\mathbb T_{\ssc}) = X_*(\mathbb T_{\ssc})$ is $Q_1$.

Next we are going to remind the reader why $Q_2$ 
takes integral values on the lattice $\Lie(\mathbb G_{\ssc})$ in
$\Lie(\mathbb G_{\ssc})_\mathbb Q$. 
The restriction of $Q_2$ to $\Lie(\mathbb T_{\ssc})$ is $Q_1$,   
which takes integral values, so
we just need to prove the integrality of the restriction 
of $Q_2$ to each subspace 
$\Lie(\mathbb G_\alpha) \oplus \Lie(\mathbb G_{-\alpha})$.  

So we need to understand what happens for $SL_2$, 
with diagonal maximal torus. Then $Q_2$ is  the quadratic form 
whose value on the matrix 
\[
\begin{bmatrix}
a & b \\
c & -a 
\end{bmatrix}
\]
is the negative of its determinant, namely $ a^2 + bc$. 
Observe that the
restriction of $Q_2$ to the direct sum of the two root spaces is
the quadratic form $bc$. We can put this answer in 
more intrinsic terms by
observing that
for $v=(x,y) \in \Lie(\mathbb G_\alpha) \oplus \Lie(\mathbb G_{-\alpha})$ the value
of $Q_2$ on $(x,y)$ is
equal to $c$, 
where $c$ is the unique integer such that the Lie bracket $[x,y]$ is equal to
$cH_\alpha$.
(We write 
$H_\alpha$ for the coroot $\alpha^\vee$ when we are viewing it in $\Lie(\mathbb
T_{\ssc})$ rather
than $X_*(\mathbb T_{\ssc})$.) Moreover $Q_2(H_\alpha)=1$ for both roots of $SL_2$.

Returning to the general case, there is still an obvious $\mathbb
T_{\ssc}$-invariant quadratic
form $Q_{\pm\alpha}$ on 
$\Lie(\mathbb G_\alpha) \oplus \Lie(\mathbb G_{-\alpha})$,
namely the one whose value
on $(x,y)$ is equal to the unique integer $c$ 
such that the Lie bracket $[x,y]$ is
equal to $cH_\alpha$.
The restriction of $Q_2$ to 
$\Lie(\mathbb G_\alpha) \oplus \Lie(\mathbb
G_{-\alpha})$ is necessarily
some scalar times  $Q_{\pm\alpha}$. Our computation 
for $SL_2$ shows that 
the scalar must be the value of $Q_2$ on $H_\alpha$, 
namely $\ell(\alpha^\vee)$. In any case $Q_2$ does 
take integral values on $\Lie(\mathbb G_{\ssc})$. 

\subsection{Definition of the nondegenerate quadratic form 
$Q_\mathbb V$ on $\mathbb V$} \label{sub.QVdef}

In the last subsection, imitating what we did in subsection
\ref{subsub.ThirdIngredient}, we
defined nondegenerate quadratic forms $Q_{\pm\alpha}$ 
for each unordered pair $\{\pm\alpha\}$.
We now define a nondegenerate quadratic form 
$Q_\mathbb V$ on $\mathbb V$ as the direct sum of
the various $Q_{\pm\alpha}$. In the next subsection 
we compare $Q_\mathbb V$ with 
the restriction of $Q_2$ to $\mathbb V$.

\subsection{The restriction of $Q_2$ to $\mathbb V$} \label{sub.Q_VvsQ_2}

Our review of how $Q_2$ is constructed gives us a 
good understanding of the restriction of $Q_2$ to $\mathbb V$.
This restriction  is equal to $Q_\mathbb V$ if and only if  
 the root system is simply laced. When the root system is not 
simply laced,    $\mathbb V$ is the orthogonal direct sum 
of $\mathbb V'$ and $\mathbb V''$, where $\mathbb V'$ is 
the direct sum of the root spaces for the various
long roots, and $\mathbb V''$ is the direct sum 
of the root spaces for the various short roots.
We denote by $Q_{\mathbb V'}$ (resp., $Q_{\mathbb V''}$) 
the restriction of $Q_{\mathbb V}$ to $\mathbb V'$ 
(resp., $\mathbb V''$). 
The computations we made while discussing $Q_2$ 
show that  
\begin{itemize}
\item
 the restriction of $Q_2$ to $\mathbb V'$ is equal to  
$Q_{\mathbb V'}$,  
\item 
 the restriction of $Q_2$ to $\mathbb V''$ is equal to  
 $\ell Q_{\mathbb V''}$. 
\end{itemize}
(Long roots correspond to short coroots, and on these $Q_2$ takes the value $1$.)

In the simply laced case it is convenient to make the convention that $\mathbb V'
=\mathbb V$ and
$\mathbb V'' = 0$. With this convention we have $\mathbb V = \mathbb V' \oplus
\mathbb V''$ in
all cases.

\section{Some lemmas about lattices in rational quadratic spaces}\label{sub.lemmas}
The next lemmas are quite general in nature. In the body of the paper we apply them
to the lattices $\Lie(\mathbb G)$, $\Lie(\mathbb T)$, $\mathbb V$, $\mathbb V'$ and
$\mathbb V''$.

\subsection{Notation}

In this appendix we consider a triple  $(\Lambda,Q,\ell)$ that consists of 
\begin{itemize}
\item
a $\mathbb Z$-lattice $\Lambda$,
\item
 a nondegenerate quadratic form 
 $Q:\Lambda_\mathbb Q \to  \mathbb Q$,   
\item 
a positive integer $\ell$, 
\end{itemize}
subject to the requirement that 
 \begin{itemize}
\item $Q$ takes integral values on $\Lambda$, and 
\item $\ell Q$ takes integral values on $\Lambda^\perp$. 
\end{itemize} 
 As always, $\Lambda^\perp$ is the perpendicular lattice 
with respect to the symmetric bilinear form $B$ obtained by polarizing $Q$, i.e. 
\[
\Lambda^\perp = \{ x \in \Lambda_\mathbb Q : B(x,y) \in \mathbb Z \quad\forall\, y
\in \Lambda \}.
\]
 Once again we remind the reader that $B(x,y)=Q(x+y)-Q(x)-Q(y)$.

\subsection{The induced quadratic forms on $\Lambda/\ell\Lambda^\perp$ and
$\Lambda^\perp/\Lambda$}

\begin{lemma} \label{lem.barQ}
The following statements hold. 
\begin{enumerate}

\item $\ell \Lambda^\perp \subset \Lambda \subset \Lambda^\perp \subset
\ell^{-1}\Lambda$. 

\item Let $\pi':\Lambda \twoheadrightarrow 
\Lambda/\ell\Lambda^\perp$ denote the canonical surjection. Then there
exists a unique quadratic form
$Q':\Lambda/\ell\Lambda^\perp \to \mathbb Z/\ell\mathbb Z$ on
the $\mathbb Z/\ell\mathbb Z$-module $\Lambda/\ell\Lambda^\perp$  such
that, for all $x \in \Lambda$, the value of $Q'$ on
$\pi'(x)$ is equal to the reduction modulo $\ell$ of $Q(x)$. 

\item Let $\pi'':\Lambda^\perp \twoheadrightarrow 
\Lambda^\perp/\Lambda$ denote the canonical surjection. Then there exists
a unique quadratic form
$Q'':\Lambda^\perp/\Lambda \to \mathbb Z/\ell\mathbb Z$ on the
$\mathbb Z/\ell\mathbb Z$-module $\Lambda^\perp/\Lambda$  such that, for
all $x \in \Lambda^\perp$, the value of $Q''$ on
$\pi''(x)$ is equal to the reduction modulo $\ell$ of $\ell Q(x)$. 

\item The quadratic form $Q$ on $\Lambda$ becomes 
nondegenerate over $\mathbb Z[1/\ell]$. 

\item If $\ell$ is prime, then the quadratic forms $Q'$ and
$Q''$ are nondegenerate over the field $\mathbb Z/\ell\mathbb
Z$.  
\end{enumerate}
\end{lemma}

\begin{proof} (1) Our assumption that $Q$ takes integral values on
$\Lambda$ implies that $B$ takes integral values on $\Lambda \times
\Lambda$, which is to say  that
$\Lambda \subset \Lambda^\perp$. Similarly, our assumption that $\ell Q$
takes integral values on $\Lambda$ implies that $\Lambda^\perp \subset
\ell^{-1}\Lambda$ (equivalently, that
$\ell \Lambda^\perp \subset \Lambda$). 

(2) We must check that $Q(x+\ell y)$ is congruent to
$Q(x)$ modulo $\ell$ for all $x \in \Lambda$, $y \in \Lambda^\perp$. This
is indeed so, because  
$Q(x+\ell y)-Q(x)=Q(\ell y)+ B(x,\ell y)$, and 
\begin{align*} Q(\ell y)&=\ell (\ell Q(y)) \in \ell \mathbb Z, \\ B(x,\ell
y)&=\ell B(x,y) \in \ell\mathbb Z.
\end{align*} 

(3) follows from (2), applied to the quadratic form $\ell Q$ and the
lattice $\Lambda^\perp$. (The perpendicular of $\Lambda^\perp$ with
respect to $\ell Q$ is $\ell^{-1}\Lambda$.) 

(4) follows from the fact that, after tensoring with $\mathbb Z[1/\ell]$,
the inclusions in (1) become equalities. 

(5) We begin by proving that $Q'$ is nondegenerate. The bilinear
form $B$ takes values in $\ell\mathbb Z$ when one of its arguments lies in
$\Lambda$ and the other in $\ell\Lambda^\perp$. Therefore $B$ induces a
symmetric bilinear form 
\[
B':\,\Lambda/\ell\Lambda^\perp \times \Lambda/\ell\Lambda^\perp
\to
\mathbb Z/\ell\mathbb Z, 
\] and it is evident that $B'$ coincides with the bilinear form
obtained by polarizing $Q'$. We must show that $B'$ is
nondegenerate. Consider an element $\bar x \in \Lambda/\ell\Lambda^\perp$
that pairs trivially with all elements in $\Lambda/\ell\Lambda^\perp$, and
represent $\bar x$ by an element $x \in \Lambda$. Then 
$B(x,y)$ lies in $\ell\mathbb Z$ for all $y \in \Lambda$, so that  
$x \in \ell\Lambda^\perp$ and hence $\bar x=0$. 

The  nondegeneracy of  $Q''$  follows from  the method 
used to derive (3) from (2). 
\end{proof}

\subsection{The homomorphism $\Aut(\Lambda,Q) \to O(Q') \times O(Q'')$} 

The next lemma concerns $\Aut(\Lambda,Q)$, the group of linear 
automorphisms of $\Lambda_\mathbb Q$ that preserve both $Q$ and $\Lambda$. 
In the lemma we assume that $\ell$ is prime. The previous lemma then provides
nondegenerate quadratic forms
\begin{itemize}
\item
 $Q'$ on  the $\mathbb F_\ell$-vector space 
$ \Lambda/\ell\Lambda^\perp$, and 
\item
$Q''$ on  the $\mathbb F_\ell$-vector space 
$\Lambda^\perp/\Lambda$.
\end{itemize}
Now suppose that we are given $g \in \Aut(\Lambda,Q)$. Then $g$ preserves $\Lambda$
and $\Lambda^\perp$,
so it induces a linear transformation $g'$ on $ \Lambda/\ell\Lambda^\perp$, as well
as a linear transformation
$g''$ on $\Lambda^\perp/\Lambda$. It is easy to see that both $g'$ and $g''$ are
orthogonal transformations,
and of course $g \mapsto (g',g'')$ is a homomorphism 
from $\Aut(\Lambda,Q)$ to $O(Q') \times O(Q'')$. In the next lemma 
we compute $g'$, $g''$ for reflections $g=r_v$ in certain vectors $v$.

\begin{lemma} \label{lem.Qv1}
Let $v \in \Lambda$ and suppose that $Q(v)$ is a unit. Consider the
reflection $g:=r_v$. Then 
\begin{itemize}
\item $g \in \Aut(\Lambda,Q)$,
\item $g'=r_{v'}$, where $v'$ denotes the image of $v$ under $\Lambda
\twoheadrightarrow \Lambda/\ell\Lambda^\perp$,
\item $g''$ is the identity element.
\end{itemize}
Similarly, now let $v \in \Lambda^\perp$ and suppose that $(\ell Q)(v)$ is a
unit. Consider the reflection $g:=r_v$. Then 
\begin{itemize}
\item $g \in \Aut(\Lambda,Q)$,
\item $g'$ is the identity element,
\item $g''=r_{v''}$, where $v''$ denotes the image of $v$ under
$\Lambda^\perp
\twoheadrightarrow \Lambda^\perp/\Lambda$.
\end{itemize}
\end{lemma}
\begin{proof}
The second half of the lemma follows easily from the first. One just has
to apply the first part with $(\Lambda,Q)$ replaced by $(\Lambda^\perp,\ell 
Q)$. It remains to prove the first half of the lemma, so we consider $v
\in \Lambda$ such that $Q(v)$ is a unit. Then $r_v$ is given by the
formula 
\[
r_v(x)=x-\frac{B(v,x)}{Q(v)} v.
\] 
Since $Q(v)$ is a unit, it is clear that $r_v$ lies in $\Aut(\Lambda,Q)$
and that $(r_v)'=r_{v'}$. It remains only to check that $(r_v)''$ is the
identity. In other words, we must show that 
$r_v(x)-x$ lies in $\Lambda$ whenever $x \in \Lambda^\perp$, and this is
clear   because
$Q(v)$ is a unit, $B(v,x)$ is integral, and $v \in \Lambda$. 
\end{proof}

\section{Comparison of Stiefel-Whitney classes before and after reduction modulo
$p$} \label{app.SWredEll}

For the purposes of this paper we just 
need to understand how Stiefel-Whitney classes behave under reduction mod $2$ and
$3$, but in this appendix
we might as well treat a general prime number $p$. Consider a nondegenerate
quadratic space $(\Lambda,Q)$ over
$\mathbb Z_p$. (When $p=2$, the rank of the free module $\Lambda$ is necessarily
even.)
Then $O(Q)$, which we abbreviate to $O$, is a smooth group scheme over $\mathbb
Z_p$.
We choose algebraic closures $\overline{\mathbb Q}_p$, $\overline{\mathbb F}_p$ of
$\mathbb Q_p$,
$\mathbb F_p$ respectively. 

Let $\Gamma$ be a profinite group, and suppose that we are given a homomorphism 
$\rho : \Gamma \to O(\mathbb Z_p)$ with open kernel. From it we obtain orthogonal
representations
$\rho_{\mathbb Q_p} : \Gamma \to O(\overline{\mathbb Q}_p)$ and 
$\rho_{\mathbb F_p} : \Gamma \to O(\overline{\mathbb F}_p)$. 

\begin{lemma} 
When $p \ne 2$, there is an equality 
\[
{SW}(\rho_{\mathbb Q_p}) = {SW}(\rho_{\mathbb F_p})
\]
of elements in $1 + H^1(\Gamma,\mathbb F_2) + H^2(\Gamma,\mathbb F_2)$. 
When $p=2$ there is an equality ${SW}_1(\rho_{\mathbb Q_p}) = 
{SW}_1(\rho_{\mathbb F_p}) \in H^1(\Gamma,\mathbb F_2)$.
\end{lemma}

\begin{proof}
We begin with the first Stiefel-Whitney class, which is obtained by applying the
sign character
$\deg:O \to \mathbb F_2$ to the given orthogonal representation. Now $\deg$ is a 
homomorphism of group schemes over $\mathbb Z_p$, so  the diagram 
\[
\begin{CD}
O(\overline{\mathbb F}_p) @<<< O(\mathbb Z_p) @>>> O(\overline{\mathbb Q}_p) \\
@V{\deg}VV @V{\deg}VV @V{\deg}VV \\
\mathbb F_2 @= \mathbb F_2 @= \mathbb F_2
\end{CD}
\] 
commutes, which proves that ${SW}_1(\rho_{\mathbb Q_p}) = 
{SW}_1(\rho_{\mathbb F_p})$. 

Now assuming that $p \ne 2$, we need to show that the
 second Stiefel-Whitney classes of
$\rho_{\mathbb Q_p}$ and $\rho_{\mathbb F_p}$ coincide. 
For this we must 
compare the double covers $\widetilde{\Pin}(k) \to O(k)$ for 
$k=\overline{\mathbb Q}_p$
and $k=\overline{\mathbb F}_p$. 
This is done using the commutative diagram 
\[
\begin{CD} 
\widetilde{\Pin}(\overline{\mathbb F}_p) @<<< \widetilde{\Pin}(\overline{\mathbb
Z}_p) @>>>
\widetilde{\Pin}(\overline{\mathbb Q}_p)  \\
@VVV @VVV @VVV \\
O(\overline{\mathbb F}_p) @<<< O(\overline{\mathbb Z}_p)
 @>>>  O(\overline{\mathbb Q}_p),  \\
\end{CD}
\] 
where $\overline{\mathbb Z}_p$ is the integral closure of $\mathbb Z_p$ in
$\overline{\mathbb Q}_p$.
The middle vertical arrow is  surjective because  
$\overline{\mathbb Z}_p$ is a strictly henselian local ring and 
$\widetilde{\Pin} \to O$ is a surjective smooth morphism 
(by our assumption that $p
\ne 2$).
So the middle arrow is a double cover whose kernel 
$\mu_2(\overline{\mathbb Z}_p)$
maps
isomorphically to the kernels $\mu_2(\overline{\mathbb F}_p)$,
$\mu_2(\overline{\mathbb Q}_p)$
of the other two double covers in the diagram.  
\end{proof}


\begin{thebibliography}{Gan98}

\bibitem[Del62]{Delzant}
Antoine Delzant, \emph{D\'efinition des classes de {S}tiefel-{W}hitney d'un
  module quadratique sur un corps de caract\'eristique diff\'erente de {$2$}},
  C. R. Acad. Sci. Paris \textbf{255} (1962), 1366--1368. \MR{0142606}

\bibitem[Del76]{Del}
Pierre Deligne, \emph{Les constantes locales de l'\'equation fonctionnelle de
  la fonction {$L$} d'{A}rtin d'une repr\'esentation orthogonale}, Invent.
  Math. \textbf{35} (1976), 299--316. \MR{0506172}

\bibitem[Fr{\"o}85]{Froh}
A.~Fr{\"o}hlich, \emph{Orthogonal representations of {G}alois groups,
  {S}tiefel-{W}hitney classes and {H}asse-{W}itt invariants}, J. Reine Angew.
  Math. \textbf{360} (1985), 84--123. \MR{799658}

\bibitem[Gan98]{Gan}
Wee~Teck Gan, \emph{A note on {K}ottwitz's invariant {$e(G)$}}, J. Algebra
  \textbf{208} (1998), no.~1, 372--377. \MR{1643946}

\bibitem[JL70]{JL}
H.~Jacquet and R.~P. Langlands, \emph{Automorphic forms on {${\rm GL}(2)$}},
  Lecture Notes in Mathematics, Vol. 114, Springer-Verlag, Berlin-New York,
  1970. \MR{0401654}

\bibitem[Kal15]{Kal}
Tasho Kaletha, \emph{Epipelagic {$L$}-packets and rectifying characters},
  Invent. Math. \textbf{202} (2015), no.~1, 1--89. \MR{3402796}

\bibitem[Kot83]{SignChanges}
Robert~E. Kottwitz, \emph{Sign changes in harmonic analysis on reductive
  groups}, Trans. Amer. Math. Soc. \textbf{278} (1983), no.~1, 289--297.
  \MR{697075}

\bibitem[Ser62]{Serre}
Jean-Pierre Serre, \emph{Corps locaux}, Publications de l'Institut de
  Math\'ematique de l'Universit\'e de Nancago, VIII, Actualit\'es Sci. Indust.,
  No. 1296. Hermann, Paris, 1962. \MR{0150130}

\bibitem[Tat79]{Tate}
J.~Tate, \emph{Number theoretic background}, Automorphic forms, representations
  and {$L$}-functions ({P}roc. {S}ympos. {P}ure {M}ath., {O}regon {S}tate
  {U}niv., {C}orvallis, {O}re., 1977), {P}art 2, Proc. Sympos. Pure Math.,
  XXXIII, Amer. Math. Soc., Providence, R.I., 1979, pp.~3--26. \MR{546607}

\bibitem[Vin95]{V}
E.~B. Vinberg, \emph{On reductive algebraic semigroups}, Lie groups and {L}ie
  algebras: {E}. {B}. {D}ynkin's {S}eminar, Amer. Math. Soc. Transl. Ser. 2,
  vol. 169, Amer. Math. Soc., Providence, RI, 1995, pp.~145--182. \MR{1364458}

\bibitem[Wal64]{Wall}
C.~T.~C. Wall, \emph{Graded {B}rauer groups}, J. Reine Angew. Math.
  \textbf{213} (1963/1964), 187--199. \MR{0167498}

\bibitem[Wei64]{Weil}
Andr{\'e} Weil, \emph{Sur certains groupes d'op\'erateurs unitaires}, Acta
  Math. \textbf{111} (1964), 143--211. \MR{0165033}

\end{thebibliography}

\providecommand{\bysame}{\leavevmode\hbox to3em{\hrulefill}\thinspace}
\providecommand{\MR}{\relax\ifhmode\unskip\space\fi MR }
\providecommand{\MRhref}[2]{%
  \href{http://www.ams.org/mathscinet-getitem?mr=#1}{#2}
}
\providecommand{\href}[2]{#2}

\end{document}